\documentclass[12pt]{amsart}

\usepackage{
	amsmath,
 	amsfonts,
  	amssymb,
 	amsthm,
        datetime,
	}

\usepackage{tikz-cd}
\usepackage{tikz}

\tikzset{Rightarrow/.style={double equal sign distance,>={Implies},->},
triple/.style={-,preaction={draw,Rightarrow}}}

\usepackage{rotating}

\usetikzlibrary{decorations.pathmorphing, decorations.pathreplacing}
 \usetikzlibrary{decorations.pathreplacing,backgrounds,decorations.markings}

\usepackage{color}

\usepackage{stmaryrd}
\usepackage[cal=boondoxo]{mathalfa}
\usepackage{mathtools}
\usepackage{bbm}

\usepackage{hyperref}
\usepackage[capitalise]{cleveref}

\usepackage{a4wide}

\usepackage{mathabx}

\usepackage{enumerate}
\usepackage{comment}

\theoremstyle{plain}
\newtheorem{thm}{Theorem}[section]
\newtheorem{cor}[thm]{Corollary}
\newtheorem{lem}[thm]{Lemma}
\newtheorem{prop}[thm]{Proposition}
\newtheorem{conj}[thm]{Conjecture}

\theoremstyle{definition}
\newtheorem{rem}[thm]{Remark}

\theoremstyle{definition}
\newtheorem{defn}[thm]{Definition}

\def\makeautorefname#1#2{\expandafter\def\csname#1autorefname\endcsname{#2}}
\makeautorefname{lem}{Lemma}%
\makeautorefname{prop}{Proposition}%
\makeautorefname{rem}{Remark}%
\makeautorefname{section}{Section}%

\newcommand{\cD}{\mathcal{D}}

\newcommand{\bN}{\mathbb{N}}

\newcommand{\bQ}{\mathbb{Q}}

\newcommand{\bZ}{\mathbb{Z}}

\newcommand{\tM}{\mathfrak{M}}

\newcommand{\bg}{{\boldsymbol{g}}}
\newcommand{\bi}{{\boldsymbol{i}}}
\newcommand{\bj}{{\boldsymbol{j}}}

\newcommand{\brak}[1]{\langle #1\rangle}
\newcommand{\pp}[1]{(\!( #1 )\!)}

\newcommand{\n}{\noindent}

\DeclareMathOperator{\Hom}{Hom}
\DeclareMathOperator{\End}{End}

\DeclareMathOperator{\Ind}{Ind}
\DeclareMathOperator{\Res}{Res}

\DeclareMathOperator{\cone}{Cone}

\newcommand{\bKO}{\boldsymbol{K}_0}

\newcommand\E{{\sf{E}}}
\newcommand\F{{\sf{F}}}

\newcommand\Le{{\sf{L}}}
\newcommand\R{{\sf{R}}}

\newcommand{\un}{\mathbbm{1}}

\newcommand{\g}{{\mathfrak{g}}}
\newcommand{\p}{{\mathfrak{p}}}
\newcommand{\bo}{{\mathfrak{b}}}
\newcommand{\lv}{{\mathfrak{l}}}
\newcommand{\np}{{\mathfrak{n}}}

\newcommand{\rp}{{\bar{\mathfrak{p}}}}
\newcommand{\gl}{{\mathfrak{gl}}}

\newcommand{\glN}{{\mathfrak{gl}_{N}}}

\newcommand{\glnn}{{\mathfrak{gl}_{2n}}}

\newcommand{\sln}{{\mathfrak{sl}_{n}}}

\DeclareMathOperator{\amod}{\mathrm{-}mod}

\DeclareMathOperator{\cl}{cl}

\DeclareMathOperator{\HKR}{HKR}

\DeclareMathOperator{\Seq}{Seq}

\definecolor{myblue}{rgb}{0,.5,1}
\definecolor{mypurple}{rgb}{.75,0,.75}
\definecolor{mygreen}{rgb}{.3,.75,.1}

\definecolor{vcolor}{rgb}{0,.5,1}

\newcommand{\tikzdiagh}[2][]{\tikz[#1,anchor= center,very thick,baseline={([yshift=1ex+#2]current bounding box.center)}]}
\newcommand{\tikzdiag}[1][]{\tikzdiagh[#1]{-1.5ex}}
\tikzstyle{tikzdot}=[fill, circle, inner sep=2pt]
\newcommand{\fdot}[1]{ \node  [anchor = center, fill=white, draw=black,circle,inner sep=2pt] at (#1) {} ;}
\newcommand{\plusspacing}{\llap{\phantom{\small ${+}1$}}}

\tikzstyle{stdhl}=[red,double=red!30,double distance=1pt]
\tikzstyle{vstdhl}=[vcolor,double=vcolor!30,double distance=1pt]
\tikzstyle{pstdhl}=[mypurple,double=mypurple!30,double distance=1pt]
\tikzstyle{nail}=[draw,fill=white!30,circle,inner sep=2pt]

\tikzstyle directed=[postaction={decorate,decoration={markings,
    mark=at position #1 with {\arrow{>}}}}]
\tikzstyle rdirected=[postaction={decorate,decoration={markings,
    mark=at position #1 with {\arrow{<}}}}]

\title[2-Verma modules and the Khovanov--Rozansky link homologies]
{2-Verma modules and the Khovanov--Rozansky link homologies}

\author{Gr\'egoire Naisse}
\address{Max-Planck Institute for Mathematics\\
 Vivatsgasse 7 \\ 
53111 Bonn\\ 
Germany}
\email{gregoire.naisse@gmail.com}

\author{Pedro Vaz}
\address{Institut de Recherche en Math\'ematique et Physique\\
Universit\'e Catholique de Louvain\\ 
Chemin du Cyclotron 2\\ 
1348 Louvain-la-Neuve\\ 
Belgium}
\email{pedro.vaz@uclouvain.be}

\begin{document}

\begin{abstract}
We explain how Queffelec--Sartori's construction of the HOMFLY-PT link polynomial can be interpreted in terms of parabolic Verma modules for $\glnn$. 
Lifting the construction to the world of categorification, we use parabolic 2-Verma modules to give a higher representation theory construction of Khovanov-Rozansky's triply graded link homology. 
\end{abstract}



\maketitle



\section{Introduction}\label{sec:intro}

One of the most celebrated homology theories of knots and links in 3-space is Khovanov and
Rozansky's $\glN$-link homology~\cite{KR1} categorifying the $\glN$-link invariant, for $N > 0$.    
Soon after the appearance of~\cite{KR1}, 
the existence of a triply-graded
link homology categorifying the HOMFLY-PT polynomial 
was predicted in~\cite{DGR} by Dunfield, Gukov and Rasmussen, who  
made various conjectures about the structure of such an homology theory.

\smallskip

A rigorous construction of a triply-graded link homology categorifying the HOMFLY-PT polynomial
was given by Khovanov and Rozansky in~\cite{KR2}
(see also~\cite{Kh} for a construction using Hochschild homology of Soergel bimodules). The structure of this link homology was studied by Rasmussen
in~\cite{rasmussen}. Rasmussen defined a family of differentials  
on the KR HOMFLY-PT link homology and showed that, for each $N>0$, 
these differentials give rise to 
a spectral sequence starting at the KR  HOMFLY-PT
homology and converging to the $\glN$-homology.

\smallskip

In this paper we use parabolic Verma modules to give a new interpretation of the HOMFLY-PT polynomial in terms of the representation theory of quantum $\gl_{2n}$.
Using the categorification of Verma modules from our previous work~\cite{naissevaz3}, we lift this procedure, yielding a new construction of KR triply-graded link homology.
We also recover a spectral sequence,  very similar to Rasmussen's~\cite{rasmussen}, from a categorical instance of the fact that we can recover irreducible, integrable representations as quotients of parabolic Verma modules.

\subsubsection*{Summary of the paper and description of the main results}

In the search for a  construction of the HOMFLY-PT polynomial based on representation theory,
Queffelec and Sartori~\cite{QS1} provided an 
algebraic gadget, called the double Schur algebra,  
which accommodates both the Hecke algebra and the Ocneanu trace under the same roof. 
The double Schur algebra $\dot{S}_{q,\beta}(\ell,k)$
contains two copies of the Schur algebra (the cases $\ell=0$ and $k=0$), and is defined as a quotient of
idempotented quantum $\mathfrak{gl}_{\ell+k}$, whose weight lattice has been shifted by a formal parameter. 
For a link $L$ presented as the closure of a braid $b$ with $n$ strands,
Queffelec and Sartori constructed an element in the double Schur algebra.  
This element is a multiple of a certain idempotent and its coefficient coincides with the HOMFLY-PT polynomial of $L$. 

\smallskip

We show that the construction in~\cite{QS1} finds a natural place in the terms of representations of 
the double Schur algebra. 
In this paper we extend the notion of Weyl modules to double Schur algebras and translate Queffelec and Sartori's
results to this context. Concretely we show that with the choice of highest weight\footnote{We allow ourselves to harmlessly abuse notation here, which will payoff further ahead.} 
$\beta=( \beta,\dotsc,\beta, 0, \dotsc ,0)$ 
(there are $n$ $\beta$'s and $n$ $0$'s in $\beta$), 
the HOMFLY-PT polynomial of $L$
can be obtained from the Weyl module $W(\beta)$ as a map which is
a multiple of the identity, the coefficient being the HOMFLY-PT polynomial of $L$.

\smallskip

As representations of $U_q(\mathfrak{gl}_{\ell+k})$, Weyl modules $W(\mu)$  
are isomorphic to parabolic Verma modules $M^\p(\mu)$ for a certain parabolic subalgebra $\p$.  
The previous paragraph can then be reformulated entirely in terms of parabolic Verma modules.  

\newtheorem*{thma}{Theorem A}
\begin{thma}[\cref{thm:homflyptMp}]
For a link presented as the closure of a braid $b$ with $n$ strands, the construction above
defines an element 
$P^\p(b)\in\End_{U_q(\glnn)}\bigl(M_{\p}( \beta ) \bigr)$  
which is a link invariant. 
It is a multiple of the identity whose coefficient equals the HOMFLY-PT polynomial of the closure of $b$. 
\end{thma}

\smallskip

For $\p$ a parabolic subalgebra of $\g$, we recall the construction of the dg-enhanced KLR algebras $R_\p$ in the form of diagrammatic algebras, as introduced in \cite{naissevaz1,naissevaz2,naissevaz3}, as well as their cyclotomic quotient $R_\p^\mu$. 
When $\p = \g$, they coincide with the usual (cyclotomic) Khovanov--Lauda and Rouquier algebras~\cite{KL1, R1}. 
Then, we explain how categories of dg-modules over $(R_\p^\mu,0)$ categorify parabolic Verma modules, with action of the quantum group given by the usual setup of induction/restriction along the map that add a vertical strand. 
We upgrade this data into a 2-category $\tM^\p(\mu)$, which we call a (parabolic) \emph{2-Verma module}.

\smallskip

The Rickard complex associated to a braid acts on the homotopy category
of complexes of the $\Hom$-categories 
of $\tM^\p(\beta)$ for a certain $\p \subset \gl_{2n}$ (this is a lift of the usual braiding induced by the embedding of
the Hecke algebra into a Schur algebra that in turn embeds canonically in a double Schur algebra). 
For a closure of a braid $b$ this procedure gives a chain complex $C(\cl(b))$.
We prove that the homotopy type of $C(\cl(b))$ is a link invariant.
This means that an isotopy of braid closures induces an isomorphism in the homotopy category of such complexes. 

\newtheorem*{thmc}{Theorem B}
\begin{thmc}[\cref{prop:CmarkovI},\cref{prop:CmarkovII},~\cref{cor:eulerchar},~\cref{thm:HisoKR}]
  The homotopy type of $C(\cl(b))$ is invariant under the Markov moves.
  Its homology groups are triply-graded link invariants and its bigraded Euler
  characteristic is the HOMFLY-PT polynomial of the closure of $b$.
Moreover, $H(b)$ is isomorphic to Khovanov and Rozansky HOMFLY-PT homology $\HKR(b)$, after regrading. 
\end{thmc}

Introducing a non-trivial differential $d_N$ on $R_\p^\beta$ turns it into another dg-algebra. 
We can then recover the usual cyclotomic quotient $R_{\glnn}^\Lambda$ of the KLR algebra
associated with $\glnn$, in the sense that the former is formal and quasi-isomorphic
to the latter. 
The work of Mackaay and Yonezawa in~\cite{MY} implies that replacing $R_\p^\beta$ by $R_{\glnn}^\Lambda$ in
the construction above produces Khovanov and Rozansky's $\glN$-link homology
$\HKR_N(\cl(b))$ for the closure of~$b$.

\smallskip 

Finally, we show that the differential $d_N$ descends to $\HKR(b)$ and engenders a spectral sequence
starting at $\HKR(b)$ and converging to $\HKR_N(b)$.

\smallskip

We have tried to keep this paper self-contained while reducing major technicalities. 
This way we hope to have made it readable by either topologists without a strong background in (higher) representation theory and by (higher) representation theorists without a strong background in topology.

\subsubsection*{Acknowledgments}
The authors would like to thank Paul Wedrich for 
 helpful discussions and for valuable comments on a preliminary version of this paper.
We also thank Jonathan Grant for 
 helpful discussions and comments on a preliminary version of this paper.
G.N. was a Research Fellow of the Fonds de la Recherche Scientifique - FNRS, under Grant no.~1.A310.16 when starting working on this project. 
G.N. is grateful to the Max Planck Institute for Mathematics in Bonn for its hospitality and financial support.
P.V. was supported by the Fonds de la Recherche Scientifique - FNRS under Grant no.~J.0135.16.




\section{Parabolic Verma modules and link invariants}\label{sec:noncat}

\subsection{Link invariants}

In~\cite{QS1} Queffelec and Sartori proposed an algebraic method to construct
the HOMFLY-PT and the Alexander polynomials of links in 3-space. 
They defined a generalization of the idempotented $q$-Schur algebra called
the \emph{doubled Schur algebra}.  
In this section we briefly recall the basics of the construction in~\cite{QS1}, 
and explain how it fits within the theory of parabolic Verma modules.

\subsubsection{The doubled Schur algebra}

In the following, we let $\beta$ be a formal parameter. We write $\lambda$ for $q^\beta$,
and work over the ring $\bZ(q,\lambda)$. 
We denote by $\Lambda_{\ell,k}^\beta$ the set of sequences
$(\mu_{-\ell+1},\dotsc ,\mu_{0},\dotsc,\mu_{k})\in (\beta-\bN_0)^{\ell}\times \bN_0^k$,  
for $\ell\geq 0$ and $k\geq 0$.

\begin{rem}
  We follow this convention, slightly different from~\cite{QS1},
because we want to relate it later with highest weight, 
rather than lowest weight, parabolic Verma modules, and
it will allow us to keep the notation simple. 
\end{rem}
 
Let $I_{\ell,k} :=\{-\ell+1,\dotsc, 0,\dotsc, k-1\}$. Let $\alpha_i:=(0,\dotsc ,0,1,-1,0,\dotsc, 0)\in\bZ^{\ell+k}$, the entry $1$ being at position
$i\in I_{\ell,k}$. For $\mu_i \in \bZ \sqcup (\bZ+\beta)$, let 
\[
  [\mu_i]_q :=
  \begin{cases}
  \dfrac{q^{\mu_i} - q ^{-\mu_i}}{q-q^{-1}}, &\text{ if } \mu_i \in \bZ,
  \\[2.5ex]
\dfrac{\lambda q^{\mu_i-\beta} - \lambda^{-1} q ^{\beta-\mu_i}}{q-q^{-1}} , &\text{ if } \mu_i \in \bZ + \beta,
\end{cases}
\]
be the (shifted) quantum number.

\begin{defn}\label{def:doubleschur}
  The doubled Schur algebra $\dot{S}_{q,\beta}(\ell,k)$ is the $\bZ(q,\lambda)$-linear 
  category defined by the following data:
\begin{itemize}
  \item Objects: $1_\mu$ for $\mu\in (\beta+\bZ)^{\ell}\times \bZ^k$, together with a zero object.
  \item Morphisms: generated by morphisms
    \[
   1_{\mu+\alpha_i}  E_i1_\mu \in\Hom(1_\mu,1_{\mu+\alpha_i}),
\mspace{30mu}
     1_{\mu-\alpha_i}F_i1_\mu \in \Hom(1_\mu,1_{\mu-\alpha_i}),
    \]
for $i\in I_{\ell,k}$, 
    together with the identity morphism of $1_\mu$ (denoted by the same symbol).
    The morphisms are subject to the following relations:
    \begin{equation}
    (E_iF_j - F_jE_i)1_\mu = \delta_{i,j}[\mu_i-\mu_{i+1}]_q1_\mu,
    \end{equation}
    \begin{align}
    E_iE_j1_\mu &= E_jE_i1_\mu, &
    F_iF_j1_\mu = F_jF_i1_\mu, &
    &\text{if $|i-j| > 1$,}
    \label{eq:EiEjcommut}
    \end{align}
    \begin{equation}
    \begin{split}
    E_i^2 E_{i\pm 1}1_\mu + E_{i\pm 1}E_i^21_\mu &= [2]_qE_iE_{i\pm 1}E_i1_\mu, \\
    F_i^2 F_{i\pm 1}1_\mu + F_{i\pm 1}F_i^21_\mu &=  [2]_qF_iF_{i\pm 1}F_i1_\mu.
    \end{split}
    \end{equation}
and $1_\mu = 0 \in\End(1_\mu,1_\mu)$ whenever $\mu\notin \Lambda_{\ell,k}^\beta$.
\end{itemize}
We often write $E_i1_\mu$ or $1_{\mu+\alpha_i}  E_i$ for $1_{\mu+\alpha_i}  E_i 1_\mu$, and similarly for $F_i$. 
\end{defn}

For $\ell=0$, we recover the idempotented $q$-Schur algebra  $\dot{S}_q(k)_d$, with  $d=\mu_1 + \dotsm + \mu_k$,
and for $k=0$, we recover $\dot{S}_q(\ell)_d$, with
$d= \ell\beta - (\mu_{-\ell+1} + \dotsm + \mu_0)$.  
There are canonical inclusions
$\dot{S}_q(k)_d\hookrightarrow \dot{S}_{q,\beta}(\ell,k)$
sending
\begin{align*}
  1_{\mu_1,\dotsc,\mu_k}&\mapsto 1_{0,\dotsc,0,\mu_1,\dotsc,\mu_k},& F_i&\mapsto F_{i+\ell+1},& E_i&\mapsto E_{i+\ell+1},
\intertext{ and
$\dot{S}_q(\ell)_d\hookrightarrow \dot{S}_{q,\beta}(\ell,k)$
sending}
1_{\mu_{-\ell+1},\dotsc,\mu_0}&\mapsto 1_{\mu_{-\ell+1},\dotsc,\mu_0,0,\dotsc,0},& F_i&\mapsto F_{i},& E_i&\mapsto E_{i} . 
\end{align*}

\subsubsection{Ladder diagrams}\label{ssec:linkinvs}

As explained in~\cite{QS1}, the doubled Schur algebra can be given a presentation in terms of ladder diagrams. 
These are generated by the ladder operators below: 

\begin{align*}
\allowdisplaybreaks 
1_\mu &\mapsto
\tikzdiag[xscale=1]{
\draw [<-] (-1,-.4) node[below]{\tiny $\mu_{-\ell+1}$} -- (-1,.5)  node[above]{\tiny $\mu_{-\ell+1}$};
\node at (0,0) {$\cdots$};
\draw [<-] (1,-.4) node[below]{\tiny $\mu_{-1}$} -- (1,.5)  node[above]{\tiny $\mu_{-1}$};
\draw [<-] (2,-.4) node[below]{\tiny $\mu_{0}$} -- (2,.5)  node[above]{\tiny $\mu_{0}$};
\draw [->] (3,-.4) node[below]{\tiny $\mu_{1}$} -- (3,.5)  node[above]{\tiny $\mu_{1}$};
\draw [->] (4,-.4) node[below]{\tiny $\mu_{2}$} -- (4,.5)  node[above]{\tiny $\mu_{2}$};
\node at (5.0,0) {$\cdots$};
\draw [->] (6,-.4) node[below]{\tiny $\mu_{k}$} -- (6,.5)  node[above]{\tiny $\mu_{k}$};
}
\\
1_{\mu-\alpha_i}F_i1_\mu &\mapsto
\tikzdiag[xscale=1]{
\draw [<-] (-1,-.4) node[below]{\tiny $\mu_{-\ell+1}$} -- (-1,.5)  node[above]{\tiny $\mu_{-\ell+1}$};
\node at (0,0) {$\cdots$};
\draw  (1,-.4) node[below]{\tiny $\mu_{i{-}1}$} -- (1,.5)  node[above]{\tiny $\mu_{i{-}1}$};
\draw  (2,-.4) node[below]{\tiny $\mu_{i}$} -- (2,.5)  node[above,xshift=-.25ex]{\tiny $\mu_{i}{-}1$};
\draw [directed=.6] (2,-.05) -- (3,.05);
\draw  (3,-.4) node[below]{\tiny $\mu_{i{+}1}$} -- (3,.5)  node[above,xshift=.25ex]{\tiny $\mu_{i{+}1}{+}1$};
\draw  (4,-.4) node[below]{\tiny $\mu_{i{+}2}$} -- (4,.5)  node[above]{\tiny $\mu_{i{+}2}$};
\node at (5.0,0) {$\cdots$};
\draw [->] (6,-.4) node[below]{\tiny $\mu_{k}$} -- (6,.5)  node[above]{\tiny $\mu_{k}$};
}
\\
1_{\mu+\alpha_i}E_i1_\mu &\mapsto
\tikzdiag[xscale=1]{
\draw [<-] (-1,-.4) node[below]{\tiny $\mu_{-\ell+1}$} -- (-1,.5)  node[above]{\tiny $\mu_{-\ell+1}$};
\node at (0,0) {$\cdots$};
\draw  (1,-.4) node[below]{\tiny $\mu_{i{-}1}$} -- (1,.5)  node[above]{\tiny $\mu_{i{-}1}$};
\draw  (2,-.4) node[below]{\tiny $\mu_{i}$} -- (2,.5)  node[above,xshift=-.25ex]{\tiny $\mu_{i}{+}1$};
\draw [directed=.6] (3,-.05) -- (2,.05);
\draw  (3,-.4) node[below]{\tiny $\mu_{i{+}1}$} -- (3,.5)  node[above,xshift=.25ex]{\tiny $\mu_{i{+}1}{-}1$};
\draw  (4,-.4) node[below]{\tiny $\mu_{i{+}2}$} -- (4,.5)  node[above]{\tiny $\mu_{i{+}2}$};
\node at (5.0,0) {$\cdots$};
\draw [->] (6,-.4) node[below]{\tiny $\mu_{k}$} -- (6,.5)  node[above]{\tiny $\mu_{k}$};
}
\end{align*}
Note that edges labeled with $\bN_0$ are oriented upwards, 
while edges labeled with $\beta-\bN_0$ are oriented downwards. 
Multiplication corresponds to concatenation of diagrams, and in our conventions
$ab$ consists of placing the diagram of $a$ on the top of the one for $b$, if the labels match, and zero otherwise.  
On the $\bZ(q,\lambda)$-space spanned by all such webs 
we impose the relations of the doubled Schur algebra
from~\cref{def:doubleschur}. In particular, thanks to \cref{eq:EiEjcommut}, we can consider such ladder diagrams up to planar isotopy exchanging the height of distant rungs. 

\medskip

We are interested by ladder diagrams where the weights $\mu_i$ are of the form $\beta$ or $\beta-1$ if $i  \leq 0$, and $0,1$ or $2$ if $i > 0$. We draw edges carrying a weight $1$ or $\beta-1$ as solid, with weight $0$ or $\beta$ as dotted, and with weight $2$ as double solid. 
 The edge forming the rungs of $E_i$ and $F_i$ are always drawn solid. For example, we have
\[
F_1F_01_{(\beta,0,0)} \mapsto
\tikzdiag[yscale=.5]{
	\draw[dotted] (0,0) node[below]{\small $\beta$} -- (0,3)  node[above]{\small $\beta{-}1$};
	\draw[dotted] (1,0) node[below]{\small $0$} -- (1,3)  node[above]{\small $0$};
	\draw[dotted] (2,0) node[below]{\small $0$} -- (2,3)  node[above]{\small $1$};
	\draw[directed=.6] (0,1) -- (1,1);
	\draw[directed=.55] (0,3) -- (0,1);
	\draw[directed=.7] (1,1) -- (1,2);
	\draw[directed=.6] (1,2) -- (2,2);
	\draw[directed=.7] (2,2) -- (2,3);
}
\]

\subsubsection{Link invariants from a doubled Schur algebra}\label{ssec:linkinvs}

We consider links presented in the form of closures of braids. 
To start with, let $b \in B_n$ be a braid diagram in $n$ strands.
For such a diagram we assign an element of $\dot{S}_{q}(n)_{n}$ using a
well known rule originally due to Lusztig~\cite[Definition 5.2.1]{lusztig}.
This extends immediately to an element of  $\dot{S}_{q,\beta}(n,n)$ 
from the embedding $\dot{S}_{q}(n)_{n}\hookrightarrow\dot{S}_{q,\beta}(n,n)$. 
Denote by $( (\beta-1)^n, (1)^n )$ the label
consisting of $n$ entries equal to $\beta-1$ and $n$ entries equal to $1$.
For $\sigma_i$ (resp. $\sigma_i^{-1}$)
a positive (resp. negative) crossing between the $i$th and the $(i+1)$th strands 
(counting from the left), 
we have: 
\begin{align*}
	\sigma_i &\mapsto -F_{i}E_{i}1_{((\beta-1)^n,(1)^n)} + q^{-1}1_{((\beta -1)^n,(1)^n)}, \\
	\sigma_i^{-1} &\mapsto  q 1_{( (\beta -1)^n,(1)^n)}- E_{i}F_{i}1_{( (\beta -1)^n,(1)^n)} .
\end{align*}
In terms of pictures, we draw it locally as:
\begin{align*}
 \tikzdiag{
	\draw[->] (1,0) -- (0,1);
	\draw[fill=white,draw=white] (.5,.5) circle (1 ex);
	\draw[->] (0,0) -- (1,1);
}
\ 
 &\mapsto 
 - \phantom{q} \  
 \tikzdiagh{-4}{
 	\draw (0,0) node[below]{\tiny $1$} -- (0,.25); \draw [double] (0,.25)  -- (0,.75); \draw[->] (0,.75) -- (0,1);
 	\draw (1,0) node[below]{\tiny $1$} -- (1,.25);\draw [dotted] (1,.25)  -- (1,.75); \draw[->] (1,.75) -- (1,1);
 	\draw[directed=.6] (1,.25) -- (0,.25);
 	\draw[directed=.6] (0,.75) -- (1,.75);
 }
 \ + 
 q^{-1}
 \ 
 \tikzdiagh{-4}{
 	\draw[->] (0,0) node[below]{\tiny $1$} -- (0,1);
 	\draw[->] (1,0) node[below]{\tiny $1$} -- (1,1);
 }
 \\
 \tikzdiag{
	\draw[->] (0,0) -- (1,1);
	\draw[fill=white,draw=white] (.5,.5) circle (1 ex);
	\draw[->] (1,0) -- (0,1);
}
\ 
 &\mapsto 
 \phantom{-}
 q
 \ 
 \tikzdiagh{-4}{
 	\draw[->] (0,0) node[below]{\tiny $1$} -- (0,1);
 	\draw[->] (1,0) node[below]{\tiny $1$} -- (1,1);
 }
 \ - \ \phantom{q^{-1}}
 \tikzdiagh{-4}{
 	\draw (0,0) node[below]{\tiny $1$} -- (0,.25); \draw [dotted] (0,.25)  -- (0,.75); \draw[->] (0,.75) -- (0,1);
 	\draw (1,0) node[below]{\tiny $1$} -- (1,.25); \draw [double] (1,.25)  -- (1,.75); \draw[->] (1,.75) -- (1,1);
 	\draw[directed=.6] (0,.25) -- (1,.25);
 	\draw[directed=.6] (1,.75) -- (0,.75);
 }
\end{align*}

\medskip

For a braid $b \in B^n$, let $\cl(b)$ be its closure on the left, 
as in the diagram below:
\[
\cl(b) :=\ 
\raisebox{-8.2ex}{
  \reflectbox{
\begin{tikzpicture}[xscale=.45,yscale=.3]
	\draw [very thick] (-1.25,-.5) rectangle (1.25,1.5);
	\draw [very thick, ->] (-1,1.5) to (-1,2.5);
	\draw [very thick, ->] (1,1.5) to (1,2.5);
	\draw [very thick, directed=0.3] (-1,-1.5) to (-1,-.5);
	\draw [very thick, directed=0.3] (1,-1.5) to (1,-.5);
	\draw [very thick, directed=0.55] (3,2.5) to (3,-1.5);
	\draw [very thick, directed=0.55] (5,2.5) to (5,-1.5);
	\draw [very thick] (1,2.5) to [out=90,in=180] (2,3.5) to [out=0,in=90] (3,2.5);
	\draw [very thick] (-1,2.5) to [out=90,in=180] (2,5.5) to [out=0,in=90] (5,2.5);
	\draw [very thick] (1,-1.5) to [out=270,in=180] (2,-2.5) to [out=0,in=270] (3,-1.5);
	\draw [very thick] (-1,-1.5) to [out=270,in=180] (2,-4.5) to [out=0,in=270] (5,-1.5);
	\node at (0,.55) {\reflectbox{$b$}};
	\node at (2,4.75) {\small $\vdots$};
	\node at (0.1,2.75) {\small $\cdots$};
	\node at (4,2.75) {\small $\cdots$};
	\node at (2,-3.25) {\small $\vdots$};
	\node at (0.1,-1.8) {\small $\cdots$};
	\node at (4,-1.8) {\small $\cdots$};
\end{tikzpicture}}}
\]
To $\cl(b)$ we assign an element of $\dot{S}_{q,\beta}(n,n)$ obtained by
adding cups and the bottom and caps at the top, using the following pattern to get a ladder diagram: 
\begin{equation}\label{eq:closingbraid}
\raisebox{-20pt}{
  \begin{tikzpicture}
\node at (-.5,0) {$\cdots$};
  \draw [very thick] (0.25,-.4) -- (0.25,.5);
  \draw [very thick,dotted] (0.25,-.85)  -- (0.25,-.4);
  \draw [very thick,dotted] (1,-.35)  -- (1,.1);
  \draw [very thick] (1,-.9)  -- (1,-.4);
  \draw [very thick,directed=.6] (0.25,-.4) -- (1,-.4);
  \draw [very thick] (1,.1)  -- (1,.5);
\draw [very thick,directed=.6] (1,.1) -- (1.75,.1);
\draw [very thick,->] (1.75,.1)  -- (1.75,.5);
\draw [very thick] (1.75,-.9)  -- (1.75,-.4);
\draw [very thick,dotted] (1.75,-.35)  -- (1.75,.1);
\draw [very thick,->] (2.5,-.4)  -- (2.5,.5);
\draw [very thick,dotted] (2.5,-.85)  -- (2.5,-.4);
\draw [very thick,directed=.6] (1.75,-.4) -- (2.5,-.4);
\draw [very thick,directed=.6] (1,-.89) -- (1.75,-.89);
\node at (3.25,0) {$\cdots$};
\end{tikzpicture}}
  \end{equation}
and similar for the top of $\cl(b)$. 

\medskip

As explained in~\cite{QS1}, this procedure gives an element $P(b)\in S_{q,\beta}(n,n)$ which is an endomorphism
of $1_{( (\beta)^n,(0)^n )}$, which in turn implies that $P(b)\in \bZ(q,\lambda)$. 
One of the main results in~\cite{QS1} is the following:

\begin{thm}[{\cite[Theorems 3.1 and 3.8]{QS1}}]\label{thm:QS}
  For a braid $b \in B_n$, the element $P(b)$ is a framed link invariant which equals the HOMFLY-PT polynomial
  of the closure of $b$. 
\end{thm}

The proof of~\cref{thm:QS} goes by first verifying that $P(b)$ is a braid invariant, and then checking
invariance under the Markov moves. In the process of showing that it equals the HOMFLY-PT polynomials,
it is shown that it gives the value $\frac{\lambda-\lambda^{-1}}{q-q-1}$ for the unknot, and it 
satisfies
\[
P\left(
\raisebox{-15pt}{
\reflectbox{
  \begin{tikzpicture}[scale=.25]
        \draw [very thick] (-1.5,-1.5)  -- (1.5,1.5);
	\draw [very thick] (-0.5,0.5) -- (-1.5,1.5);
	\draw [very thick] (1.5,-1.5)  -- (0.5,-0.5);
        \draw [very thick,directed=.6] (3.5,1.5)  -- (3.5,-1.5);
	\draw [very thick] (3.5,-1.5)  -- (1.5,-1.5);
        \draw [very thick] (3.5, 1.5)  -- (1.5, 1.5);
 \draw [very thick, dotted] (1.5,1.8) -- (1.5,2.4);
 \draw [very thick, dotted] (3.5,1.8) -- (3.5,2.4);
 \draw [very thick,->] (-1.5,1.5) -- (-1.5,2.4);
 \draw [very thick, dotted] (1.5,-1.8) -- (1.5,-2.4);
 \draw [very thick, dotted] (3.5,-1.8) -- (3.5,-2.4);
 \draw [very thick] (-1.5,-1.5) -- (-1.5,-2.4);
\end{tikzpicture}}}\
\right)
=\lambda 
P\left(
\raisebox{-15pt}{
  \reflectbox{
\begin{tikzpicture}[scale=.25]
	\draw [very thick,->] (-1.5,-2.4) -- (-1.5,2.4);
	\draw [very thick,dotted] (1,-2.4) -- (1,2.4);
       	\draw [very thick,dotted] (3,-2.4) -- (3,2.4);
\end{tikzpicture}}}\
\right),
\qquad
P\left(
\raisebox{-15pt}{
  \reflectbox{
\begin{tikzpicture}[scale=.25]
        \draw [very thick] (-1.5,-1.5)  -- (-0.5,-0.5);
	\draw [very thick] (0.5,0.5) -- (1.5,1.5);
	\draw [very thick] (1.5,-1.5)  -- (-1.5,1.5);
	\draw [very thick,directed=.6] (3.5,1.5)  -- (3.5,-1.5);
	\draw [very thick] (3.5,-1.5)  -- (1.5,-1.5);
        \draw [very thick] (3.5, 1.5)  -- (1.5, 1.5);
 \draw [very thick, dotted] (1.5,1.8) -- (1.5,2.4);
 \draw [very thick, dotted] (3.5,1.8) -- (3.5,2.4);
 \draw [very thick,->] (-1.5,1.5) -- (-1.5,2.4);
 \draw [very thick, dotted] (1.5,-1.8) -- (1.5,-2.4);
 \draw [very thick, dotted] (3.5,-1.8) -- (3.5,-2.4);
 \draw [very thick] (-1.5,-1.5) -- (-1.5,-2.4);
\end{tikzpicture}}}\
\right)
=\lambda^{-1}\ 
P\left(
\raisebox{-15pt}{
  \reflectbox{
\begin{tikzpicture}[scale=.25]
	\draw [very thick,->] (-1.5,-2.4) -- (-1.5,2.4);
	\draw [very thick,dotted] (1,-2.4) -- (1,2.4);
       	\draw [very thick,dotted] (3,-2.4) -- (3,2.4);
\end{tikzpicture}}}\
\right) ,
\]
and the skein relation
\[
P\left(
\raisebox{-10pt}{
\begin{tikzpicture}[scale=.3]
        \draw [very thick, ->] (-1.5,-1.5)  -- (1.5,1.5);
	\draw [very thick, ->] (-0.5,0.5) -- (-1.5,1.5);
	\draw [very thick] (1.5,-1.5)  -- (0.5,-0.5);
\end{tikzpicture}} \ 
\right)
-
P\left(
\raisebox{-10pt}{
\begin{tikzpicture}[scale=.3]
         \draw [very thick] (-1.5,-1.5)  -- (-0.5,-0.5);
	\draw [very thick, ->] (0.5,0.5) -- (1.5,1.5);
	\draw [very thick,->] (1.5,-1.5)  -- (-1.5,1.5);
\end{tikzpicture}} \ 
\right)
=
(q^{-1} - q)
P\left(
 \raisebox{-10pt}{
\begin{tikzpicture}[scale=.3]
	\draw [very thick,->] (-1.5,-1.5) -- (-1.5,1.5);
	\draw [very thick,->] (1.5,-1.5)  -- ( 1.5,1.5);
\end{tikzpicture}}\ 
\right) .
\]
Multiplying $P(b)$ by $\lambda^{w(\cl{b})}$, where $w(\cl{b})$
is the writhe of $\cl(b)$,
results in the usual, framing independent, HOMFLY-PT polynomial. 

\medskip

By the usual specializations of $\lambda$, we recover the $\glN$-polynomial  (for $\lambda=q^N$)  and
the Alexander polynomial (for $\lambda=1$) of the closure of $b$.
Note that for the latter one needs to cut open one of the strands to avoid getting the value
zero associated to the unknot, and therefore to any link, as explained in~\cite[\S4]{QS1}
(see the discussion on normalized invariants in~\cref{ssec:normalized} below for further details).

\subsubsection{Weyl modules}\label{ssec:weylmodules}


We introduce a partial order $\preceq$ on $\Lambda_{\ell,k}^\beta$ by declaring that $\nu \preceq \mu$ whenever
\begin{align*}
\nu_i - \mu_i \leq 0,& &\text{for all $i \leq 0$,}  \\
\nu_i - \mu_i \geq 0,& &\text{for all $i > 0$.} 
\end{align*}
Let $\Lambda_+ := \{\mu \in \Lambda_{\ell,k}^\beta | \mu_i - \mu_{i+1} \geq 0 \text{ for all } i\in I_{k,\ell} \setminus\{0\} \}$. 

\begin{defn}
For $\mu\in \Lambda_+$, we define the \emph{Weyl module}
\[ 
W(\mu) := \dot{S}_{q,\beta}(\ell,k)1_\mu / (\nu\npreceq\mu).  
\] 
Here $(\nu\npreceq\mu)$ is the left ideal generated by all elements of the form $1_\nu x1_\mu$, for some
$x\in\dot{S}_{q,\beta}(\ell,k)$ and $\nu \npreceq \mu$. 
\end{defn}

For $\ell=0$, we recover the well-known Weyl modules for the $q$-Schur algebra
 $\dot{S}_q(k)_{\mu_1+\dotsm + \mu_k}$. 
As in the case of $\ell=0$, it is also true that $U_q(\mathfrak{gl}_{k+\ell})$ acts on $W(\mu)$: 
for $F_i\in U_q(\mathfrak{gl}_{k+\ell})$ and $1_{\nu}x1_\mu \in W(\mu)$ we put
$F_i \cdot 1_{\nu}x1_\mu := 1_{\nu -\alpha_i}F_ix1_\mu$, and similarly for $E_i$.
Note that the Chevalley generators of $U_q(\mathfrak{gl}_{k+\ell})$ are indexed from
$\{-\ell+1, \dotsc , k-1\}$, which is the set $I_{\ell,k}$
introduced in the definition of $\dot{S}_{q,\beta}(\ell,k)$. 

\smallskip

Note that $E_i 1\mu = 0 \in W(\mu)$ for all $i \in I_{k,\ell}$. Thus, $1_\mu \in W(\mu)$ is a highest weight object, and $\End_{W(\mu)}(1_\mu) \cong \bZ(q,\lambda)$.
From~\cref{ssec:linkinvs}, all weights occurring in $P(b)\in\dot{S}_{q,\beta}(n,n)$ are of the form $\nu \preceq  ((\beta)^n,(0)^n)$ since the only weights appearing are $\beta,\beta-1$ and $0,1,2$.
Therefore, $P(b)$ is
sent to the same word in $E$'s and $F$'s
under the quotient map $\dot{S}_{q,\beta}(n,n)\to W( (\beta)^n,(0)^n )$.
In particular, the results in \cite[Theorem 3.8]{QS1} imply the following:

\begin{prop}\label{prop:PbonWeyl}
The element $\lambda^{w(\cl(b))}P(b)$ acts on $W( (\beta)^n,(0)^n )$ as an endomorphism of the highest weight 
object, which is multiplication by the HOMFLY-PT polynomial of the closure of~$b$.
\end{prop}

\subsection{Parabolic Verma modules}\label{sec:paraverma}

Consider $\g = \gl_m$ with simple roots $I_\g = \{\alpha_0, \dots, \alpha_{m-1}\}$. Let $\Upsilon_\g = \{0, \dots, m\}$, and define
\[
\alpha_i(\gamma) := 
\begin{cases}
	1, &\text{if $\gamma=i$}, \\
	-1, &\text{if $\gamma=i+1$}, \\
	0, &\text{otherwise},
\end{cases}
\]
for all $\gamma \in \Upsilon_\g$ and $i \in I_\g$.

Recall that the quantum group $U_q(\g)$ the $\bQ(q)$-algebra generated by the \emph{Chevalley generators} $E_i,F_i$ for all $i \in I_\g$ and the \emph{Cartan elements} $K_\gamma^{\pm 1}$ for all $\gamma \in \Upsilon_\g$, with relations
\begin{align*}
K_\gamma K_\gamma^{-1}  &= 1 = K_\gamma^{-1}K_\gamma,
& K_{\gamma} K_{\gamma'} &=K_{\gamma'} K_{\gamma}, \\
K_{\gamma} E_i  &= q^{\alpha_i(\gamma)} E_i K_{\gamma},  &
 K_{\gamma} F_i &= q^{\alpha_i(\gamma)} F_i K_{\gamma},
\end{align*}
\begin{equation*}
E_iF_j - F_jE_i = \delta_{ij} \frac{\tilde K_i - \tilde K_i^{-1}}{q - q^{-1}},
\end{equation*}
where $\tilde K_i := K_iK_{i+1}^{-1}$,
\begin{align*}
    E_i^2 E_{i\pm 1} + E_{i\pm 1}E_i^2 &= [2]_qE_iE_{i\pm 1}E_i, 
    &
    E_iE_j &= E_jE_i & \text{if $|i-j|>1,$} 
     \\
    F_i^2 F_{i\pm 1} + F_{i\pm 1}F_i^2 &=  [2]_qF_iF_{i\pm 1}F_i,
    &
    F_iF_j &= F_jF_i & \text{if $|i-j|>1,$} 
\end{align*}
for all $i,j \in I_\g$ and $\gamma, \gamma' \in \Upsilon_\g$. 

 \smallskip

The (standard) \emph{Borel subalgebra $U_q(\bo)$} is the $U_q(\g)$-subalgebra generated by $\brak{E_i, K_i}_{i \in I_\g}$. 
A (standard) \emph{parabolic subalgebra $U_q(\p)$} is an $U_q(\g)$-subalgebra such that $U_q(\bo) \subset U_q(\p)$. 
For any subset of simple roots $I_\lv \subset I_\g$, we can define a parabolic subalgebra $U_q(\p) = \brak{E_i, F_j, K_\gamma}_{i \in I_\g, j \in I_\lv, \gamma \in \Upsilon_\g}$. As a matter of fact, any parabolic subalgebra is of this form for some choice of $I_\lv$. 
 The subalgebra $U_q(\lv) = \brak{E_j,F_j, K_\gamma}_{ j \in I_\lv, \gamma \in \Upsilon_\g}$ is called the \emph{Levi factor}, and the part $U_q(\mathfrak{n}) = \brak{E_i}_{i \in I_{\mathfrak n} := I_\g\setminus I_\lv}$ is the \emph{nilpotent radical}.
 
 \smallskip
 
 Fix a parabolic subalgebra given by $I_\lv$. We choose a weight $\mu = \{\mu_\gamma\}_{\gamma \in \Upsilon_\g}$ such that $\mu_i - \mu_{i+1} \in \bN_0$ for each $i \in I_\lv$, and $\mu_j - \mu_{j+1} \in \beta + \bZ$ for each $j \in I_\np$. 
 There is a unique irreducible, integrable $U_q(\lv)$-module $L(\mu)$ over $\bQ(q,\lambda = q^\beta)$ with highest weight $\mu$. This means $L(\mu)$ is generated by a highest weight vector $v_\mu$ such that
\begin{align*}
K_\gamma v_\mu &= q^{\mu_\gamma} v_\mu,
&
E_i v_\mu &= 0,
\end{align*}
for all $i \in I_\g$ and $\gamma \in \Upsilon_g$. We extend $V(\mu)$ to a $U_q(\p)$-module by setting $U_q(\np)L(\mu) := 0$. 

\begin{defn}
The \emph{parabolic Verma module} with highest $\mu$ is the induced module
\[
M^\p(\mu) := U_q(\g) \otimes_{U_q(\p)} L(\mu). 
\]
\end{defn}

When $\p$ coincides with the Borel subalgebra $\mathfrak{b}$, we recover the usual Verma module. When $\p = \g$, then we get the irreducible, integrable representation $L(\mu)$. 
See~\cite[Chapter 9]{humphreysO} for further details on parabolic Verma modules,
and~\cite{mazorchuk1} (and references therein) for a detailed study of parabolic Verma modules.

\subsection{Parabolic Verma modules and link invariants}\label{ssec:linkpverma}

We consider $\g = \gl_{\ell+k}$, and we identify $I_\g$ with $I_{\ell,k}$, so that the Chevalley generators are indexed by elements in $I_{\ell,k}$. 
Consider the parabolic subalgebra given by $I_\lv = I_{\ell,k} \setminus\{0\}$. 
Thus, its nilpotent radical is generated by $E_0$.

\begin{prop}\label{prop:isowpar}
For $\mu \in \Lambda_+$, the module $W(\mu)$ is isomorphic to the parabolic Verma module $M^\p(\mu)$
as modules over $U_q(\mathfrak{gl}_{\ell+k})$. 
\end{prop}
\begin{proof}
As mentioned above, $U_q(\mathfrak{gl}_{k+\ell})$ acts on $W(\mu)$ which in particular is a weight module. 
Moreover, $1_\mu$ is a highest weight vector and since $W(\mu)$ is cyclic generated by $1_\mu$, $W(\mu)$ is a highest weight module. 
We see that $W(\mu)$ is a Verma module and there is a surjection $M^\bo(\mu) \twoheadrightarrow W(\mu)$. 
By \cite[Theorem 1.2]{humphreysO} there are finitely many highest weight modules for a fixed highest weight, up to isomorphism, and they are given by all the parabolic Verma modules (including the cases $\p = \g$ and $\p = \bo$). Thus, it is enough to study the nilpotency of the operator $F_i$ for all simple root $\alpha_i$. 
By the PBW basis theorem \cite[Proposition 4.16]{jantzen96} of $U_q(\mathfrak{gl}_{k+\ell})$, we know that $\F_0^k 1_\mu \neq 0$ for all $k \geq 0$. 
One can also  see that $\F_{\pm i}$ for $i \neq 0$ acts locally nilpotently on $W(\mu)$.  
Indeed for $i > 0$, given $\nu \preceq \mu$, we have $\nu_i = \mu_i + k$ for some $k \geq 0$, and thus $\F_i^{k+1} 1_\nu = 0$. 
Similarly, for $i < 0$ one has $\nu_{i+1} = \mu_{i+1} - k$ for some $k \geq 0$, and thus $\F_i^{k+1} 1_\nu = 0$. 
Therefore, we conclude that $W(\mu)$ is isomorphic to the parabolic Verma module $M^\p(\mu)$.
\end{proof}

\n{\bf Notation.} From now on, for the sake of keeping the notation simple we denote the highest
weight modules $W( (\beta)^n,(0)^n )$ and $M^{\p}( (\beta)^n,(0)^n )$ 
by $W( \beta )$ and $M^{\p}( \beta )$ respectively. 

\smallskip
In the particular case of $M^{\p}( \beta )$, the irreducible $L( (\beta)^n,(0)^n )$ is $1$-dimensional. 
Under the isomorphism in~\cref{prop:isowpar}, the element $P(b)$ defines an endomorphism 
$P^\p(b)$ of the highest weight object of the Verma module $M^{\p}( \beta ) $ (seen as a linear category with objects indexed by the weights, in the obvious way). 
Since $P^\p(b)$ consists of the same word in $E$'s and $F$'s
as $P(b)$, it yields the same element in $\bQ(q, \lambda)$. 
Thus, \cref{prop:PbonWeyl} translates to the following theorem:

\begin{thm}\label{thm:homflyptMp}
  For a braid $b \in B_n$, the element
  $\lambda^{w(\cl(b))}P^\p(b)\in \End_{U_q(\glnn)}\bigl( M^{\p} (\beta) \bigr)$
 is a link invariant which equals the HOMFLY-PT polynomial
  of the closure of $b$. 
\end{thm}

Taking the whole algebra $U_q(\g)$ as parabolic subalgebra and a highest weight $((N)^n,(0)^n)$ instead, the element $\lambda^{w(\cl(b))}P^\g(b)$ gives an endormorphism of the highest weight object of $L(N)$, which coincides with multiplication by the $\gl_N$-polynomial. This is a well-know result that can be explained through quantum skew-Howe duality \cite{qSkewHowe}.

\subsubsection{Normalized link invariants}\label{ssec:normalized}

In order to be able to compute the normalized HOMFLY-PT and $\glN$-link 
invariants, we follow the procedure described in~\cite[\S4]{QS1}.
We denote $\cl_o(b)$ the diagram obtained by closing all but the outermost strands of $b$.
We can also think of it as cut opening the braid closure 
diagram $\cl(b)$ into a special type of $(1,1)$-tangle diagram. 
More precisely, we open the diagram of  $\cl(b)$ by cutting the outermost strand, following a pattern as 
shown in the example in~\cref{eq:openclosedbraid} below for a braid with three strands, 
 \begin{equation}\label{eq:openclosedbraid}
   \raisebox{-28pt}{
     \reflectbox{
  \begin{tikzpicture}
\draw [very thick,->] (-0.5,-.9) -- (-0.5,.5);
\draw [very thick,directed=.6] (.25,-.9) -- (-0.5,-.9);
\draw [very thick,dotted] (-.5,-1.9)  -- (-.5,-.9);
\draw [very thick] (.25,-1.4)  -- (.25,-.9);
\draw [very thick,dotted] (1,-1.4)  -- (1,-.9);
\draw [very thick,dotted] (0.25,-1.9)  -- (0.25,-1.4);
\draw [very thick,directed=.6] (1,-1.4) -- (.25,-1.4);
\draw [very thick,directed=.6] (1,-1.9)  -- (1,-1.4);
\draw [very thick,dotted] (1.75,-1.9)  -- (1.75,-0.9);
\draw [very thick,->] (0.25,-.4) -- (0.25,.5);
  \draw [very thick,dotted] (0.25,-.85)  -- (0.25,-.4);
  \draw [very thick,dotted] (1,-.35)  -- (1,.1);
  \draw [very thick] (1,-.9)  -- (1,-.4);
  \draw [very thick,directed=.6] (1,-.4) -- (0.25,-.4);
  \draw [very thick,->] (1,.1)  -- (1,.5);
\draw [very thick,directed=.6] (1.75,.1) -- (1,.1);
\draw [very thick] (1.75,.1)  -- (1.75,.5);
\draw [very thick] (1.75,-.9)  -- (1.75,-.4);
\draw [very thick,dotted] (1.75,-.35)  -- (1.75,.1);
\draw [very thick] (2.5,-.4)  -- (2.5,.5);
\draw [very thick,dotted] (2.5,-1.9)  -- (2.5,-.4);
\draw [very thick,directed=.6] (2.5,-.4) -- (1.75,-.4);
\draw [very thick,directed=.6] (1.75,-.89) -- (1,-.89);
\end{tikzpicture}}}
  \end{equation}
and similarly for the top part.  

\smallskip

The procedure described in~\cref{ssec:linkinvs}
gives an element $\overline{P}(b)\in S_{q,\beta}(n-1,n)$ which is an endomorphism 
of $1_{( (\beta)^{n-1},1,(0)^{n-1} )}$, which in turn implies that $\overline{P}(b)\in \bZ(q,\lambda)$  
(see~\cite[\S4]{QS1} for details).

\begin{thm}[{\cite[Proposition 4.6]{QS1}}]\label{thm:QS-homflypt}
  For a braid $b \in B_n$, the element $\overline{P}(b)$ is a framed link invariant
  which equals the reduced HOMFLY-PT polynomial of the closure of $b$. 
\end{thm}

Note that we could have opened the diagram in a different way, by choosing a different strand to cut it open.    
We could have equally opened the diagram by cutting it using one of the inner strands at the expense of
adding crossings to the original diagram.  
In~\cite{QS1}, it is proven that the link invariant obtained does not depend on this choice.   

\smallskip

In order to parallel the construction of~\cref{ssec:linkpverma} using a parabolic Verma module, we consider
$\mathfrak{gl}_{2n-1}$ with simple roots
$\{ \alpha_{2-n}, \dots, \alpha_{n-1} \}$ (we no longer need the root
$\alpha_{1-n}$ since the braid is not completely closed on the left anymore). 
We form the parabolic subalgebra $U_q(\rp)$ given by $I_\lv = I_\g \setminus\{0\}$.
Then, we consider the parabolic Verma module 
\[
M^{\rp}( \overline\beta ) = M^{\rp}( (\beta)^{n-1},1,(0)^{n-1} ),
\]
where the highest weight is chosen to agree with the bottom of~\cref{eq:openclosedbraid}:
   \[
  \begin{tikzpicture}
\draw [very thick,dotted] (-.5,-1.9) node[below] {$\beta$} -- (-.5,-1.4);
\draw [very thick,dotted] (0.25,-1.9)node[below] {$\beta$}  -- (0.25,-1.4);
\draw [very thick,directed=.6] (1,-1.9)node[below] {$1$}  -- (1,-1.4);
\draw [very thick,dotted] (1.75,-1.9)node[below] {$0$}  -- (1.75,-1.4);
\draw [very thick,dotted] (2.5,-1.9)node[below] {$0$}  -- (2.5,-1.4);
\end{tikzpicture}
\]


The method described in~\cref{ssec:linkpverma} defines an endomorphism 
$\overline{P}^{\rp}(b)$ of the highest weight object of $M^{\rp}( \overline\beta)$.
The following is an immediate consequence of the paragraphs above: 
\begin{thm}\label{prop:isopwpar2}
  For a braid $b$ the element
  $\lambda^{w(\cl_o(b))}\overline{P}^{\rp}(b)\in \End_{U_q(\mathfrak{gl}_{2n-1})}\bigl(M^{\rp}( \overline\beta ) \bigr)$
 is a link invariant which equals the reduced HOMFLY-PT polynomial of the closure of $b$. 
\end{thm}

\begin{rem}
From now on, for the means of higher representation theory, we will consider the parabolic Verma modules $M^\p$ and $M^{\rp}( \overline\beta )$ over the ground field $\bQ\pp{q, \lambda} \supset \bQ(q,\lambda)$ with polynomial fractions viewed as formal power series. See \cite[\S5.3]{naissevaz1} for more about rings of formal Laurent series in the context of categorification, see also \cite{laurent} for a general discussion about these rings.
\end{rem}

%
%



\section{Parabolic 2-Verma modules}\label{sec:twopVerma}

We recall the construction of parabolic 2-Verma modules (i.e. categorified Verma modules) from \cite{naissevaz3}, using dg-enhanced KLR algebras.

We fix a unital commutative ring $\Bbbk$. Also, in our convention, a $\bZ^n$-graded dg-$\Bbbk$-algebra $(A,d)$, where $A = \bigoplus_{(h,\bg) \in \bZ \times \bZ^n} A_\bg^h$, is a dg-algebra carrying an extra $\bZ^n$-grading and having a differential of degree $-1$ w.r.t. the homological grading and that preserves the $\bZ^n$-grading: $d(A_\bg^h) \subset A_{\bg}^{h-1}$. 

\subsection{Dg-enhanced KLR algebras}

Fix a parabolic subalgebra $\p$ of $\gl_{n}$, obtained from a subset of simple roots $I_\lv \subset I_\g$. 

\begin{defn}
The $\p$-KLR algebra $R_\p(m)$ on $m$ strands is the diagrammatic $\Bbbk$-algebra where elements are $\Bbbk$-linear combinations of braid-like diagrams on $m$-strands, read from bottom to top, such that:
\begin{itemize}
\item strands are labeled by a simple root in $I_\g$;
\item two strands can only intersect transversely;
\item strands can be decorated by dots;
\item multiplication is given by gluing diagrams on top of each other, where $ab$ means we put $a$ on top of $b$, if the labels of the strands agree, and is zero otherwise;
\item the region immediately at the right of the left-most strand can be decorated with a floating dot (written as a hollow dot), if its label is in $I_\np$:
\[
\tikzdiagh[xscale=.75]{0}{
	\fdot{0.5,0.5};
	\draw (0,0) node[below] {\plusspacing \small $i_1$} -- (0,1);
	\draw (1,0) node[below] {\plusspacing \small $i_2$} -- (1,1);
	\node at(2,.5) {$\dots$};
	\draw (3,0) node[below] {\plusspacing \small $i_{m}$} -- (3,1);
}
\]
for $i_1 \in I_\np$;
\item diagrams are taken modulo braid-like planar isotopy and the following local relations:
  \begin{align}\label{eq:KLRR2}
\tikzdiagh{0}{
	      	\draw  (0,-.75) node[below] {\small $i$} .. controls (0,-.375) and (1,-.375) .. (1,0) .. controls (1,.375) and (0, .375) .. (0,.75);
 	 	\draw[myblue]  (1,-.75) node[below] {\small $j$} .. controls (1,-.375) and (0,-.375) .. (0,0) .. controls (0,.375) and (1, .375) .. (1,.75);
}
\ = \ 
\begin{cases}
	\hfil 0 &\text{ if } i = j, \\ \\
	\mspace{42mu}\tikzdiagh{0}{
      		\draw  (0,-.5) node[below] {\small $i$} -- (0,.5);
		\draw[myblue]  (1,-.5) node[below] {\small $j$} -- (1,.5);
	} &\text{ if } |i  - j| > 1,\\
           \tikzdiagh{0}{
      		\draw  (0,-.5) node[below] {\small $i$} -- (0,.5)node [midway,tikzdot]{};
		\draw[myblue]  (1,-.5) node[below] {\small $j$} -- (1,.5);
	}
	\ + \ 
	\tikzdiagh{0}{
      		\draw  (0,-.5) node[below] {\small $i$} -- (0,.5) ;
		\draw[myblue]  (1,-.5) node[below] {\small $j$} -- (1,.5)node [midway,tikzdot]{} ;
	}  &\text{ if } |i-j| = 1,
\end{cases} 
\end{align}
for all $i,j \in I_\g$,
\begin{align}\label{eq:KLRdotslide}
	\tikzdiagh{0}{
	          \draw (0,-.5) node[below] {\small $i$} .. controls (0,0) and (1,0) .. (1,.5);
	          \draw[myblue] (1,-.5) node[below] {\small $j$} .. controls (1,0) and (0,0) .. (0,.5)  node [near end,tikzdot]{};
	} 
	&\  = \ 
	\tikzdiagh{0}{
	          \draw (0,-.5) node[below] {\small $i$} .. controls (0,0) and (1,0) ..  (1,.5);
	          \draw[myblue] (1,-.5) node[below] {\small $j$}  .. controls (1,0) and (0,0) ..  (0,.5) node [near start,tikzdot]{};
	} 
&
	\tikzdiagh{0}{
	          \draw (0,-.5) node[below] {\small $i$} .. controls (0,0) and (1,0) ..  (1,.5) node [near start,tikzdot]{};
	          \draw[myblue] (1,-.5) node[below] {\small $j$} .. controls (1,0) and (0,0) ..  (0,.5);
	} 
	&\  = \ 
	\tikzdiagh{0}{
	          \draw (0,-.5) node[below] {\small $i$} .. controls (0,0) and (1,0) ..  (1,.5)node [near end,tikzdot]{};
	          \draw[myblue] (1,-.5) node[below] {\small $j$} .. controls (1,0) and (0,0) ..  (0,.5);
	} 
\\
\label{eq:KLRnh}
	\tikzdiagh{0}{
	          \draw (0,-.5) node[below] {\small $i$} .. controls (0,0) and (1,0) ..  (1,.5);
	          \draw (1,-.5) node[below] {\small $i$} .. controls (1,0) and (0,0) ..  (0,.5)  node [near end,tikzdot]{};
	}  
	&\  = \ 
	\tikzdiagh{0}{
	          \draw (0,-.5) node[below] {\small $i$} .. controls (0,0) and (1,0) ..  (1,.5); 
	          \draw (1,-.5) node[below] {\small $i$} .. controls (1,0) and (0,0) ..  (0,.5) node [near start,tikzdot]{};
	} \  + \  
	\tikzdiagh{0}{
	          \draw (0,-.5) node[below] {\small $i$} -- (0,.5);
	          \draw (1,-.5) node[below] {\small $i$} -- (1,.5);
	} ,
&
	\tikzdiagh{0}{
	          \draw (0,-.5) node[below] {\small $i$} .. controls (0,0) and (1,0) ..  (1,.5) node [near start,tikzdot]{};
	          \draw (1,-.5) node[below] {\small $i$}  .. controls (1,0) and (0,0) ..  (0,.5);
	} 
	&\  = \ 
	\tikzdiagh{0}{
	          \draw (0,-.5) node[below] {\small $i$} .. controls (0,0) and (1,0) ..  (1,.5)node [near end,tikzdot]{};
	          \draw (1,-.5) node[below] {\small $i$}  .. controls (1,0) and (0,0) ..  (0,.5);
	} \  + \ 
	\tikzdiagh{0}{
	          \draw (0,-.5) node[below] {\small $i$} -- (0,.5);
	          \draw (1,-.5) node[below] {\small $i$} -- (1,.5);
	} 
\end{align}
for all $i \neq j \in I_\g$, 
\begin{align}
 \label{eq:KLRR3} 
	\tikzdiagh[scale=.75]{0}{
		\draw  (0,0)node[below] {\small $i$} .. controls (0,0.5) and (2, 1) ..  (2,2);
		\draw[mygreen]  (2,0)node[below] {\small $k$} .. controls (2,1) and (0, 1.5) ..  (0,2);
		\draw[myblue]  (1,0)node[below] {\small $j$} .. controls (1,0.5) and (0, 0.5) ..  (0,1) .. controls (0,1.5) and (1, 1.5) ..  (1,2);
	 }  \  - \  
	\tikzdiagh[scale=.75]{0}{
		\draw  (0,0)node[below] {\small $i$} .. controls (0,1) and (2, 1.5) ..  (2,2);
		\draw[mygreen]  (2,0)node[below] {\small $k$} .. controls (2,.5) and (0, 1) ..  (0,2);
		\draw[myblue]  (1,0)node[below] {\small $j$} .. controls (1,0.5) and (2, 0.5) ..  (2,1) .. controls (2,1.5) and (1, 1.5) ..  (1,2);
	 }
\  &= 
\begin{cases}
\hfil 0 &\text{ if } i \neq k \text{ or } |i-j| > 1, \\ \\
 \ 
\tikzdiagh[scale=.75]{0}{
	     	 \draw[myblue]  (1,-1) node[below] {\small $j$}  --(1,1); 
	      	\draw  (0,-1) node[below] {\small $i$} -- (0,1); 
	      	\draw  (2,-1) node[below] {\small $i$} -- (2,1); 	
	 }
\quad& \text{otherwise,} 
\end{cases}
\end{align}
for all $i, j, k \in I_\g$,
\begin{align}
\tikzdiagh[yscale=1.5]{0}{
	\draw[myblue] (0,0) node[below]{\small $j$} ..controls (0,.15) and (1,.15) .. (1,.5)
		 ..controls (1,.85) and (0,.85) .. (0,1) -- (0,1.5);
	\draw (1,0) node[below]{\small $i$} ..controls (1,.15) and (0,.15) .. (0,.5)
		..controls (0,.85) and (1,.85) .. (1,1) -- (1,1.5);
	\fdot{.5,.5};
	\fdot{.5,1.25};
}
\ + \ 
\tikzdiagh[yscale=1.5]{0}{
	\draw[myblue] (0,-.5) node[below]{\small $j$} -- (0,0)  ..controls (0,.15) and (1,.15) .. (1,.5)
		 ..controls (1,.85) and (0,.85) .. (0,1);
	\draw (1,-.5) node[below]{\small $i$}  -- (1,0)..controls (1,.15) and (0,.15) .. (0,.5)
		..controls (0,.85) and (1,.85) .. (1,1);
	\fdot{.5,.5};
	\fdot{.5,-.25};
}
\ &= \ 0,
&
\tikzdiagh[yscale=1.5]{0}{
	\draw (0,0)  node[below]{\small $i$} -- (0,1);
	\fdot{.5,.25};
	\fdot{.5,.75};
}
\ &= 0,
\end{align}
for all $i,j \in I_\np$.
\end{itemize}
\end{defn}
Note that $R_\g(m)$ is exactly the usual KLR algebra, as defined in~\cite{KL1,R1}. 

\smallskip

As the cyclotomic quotients of  KLR algebras categorify the irreducible, integrable modules, certain quotients of the $\p$-KLR algebras categorify the parabolic Verma modules. Fix a weight $\mu$ as in \cref{sec:paraverma}. 

\begin{defn}
The \emph{$\mu$-cyclotomic $\p$-KLR algebra} $R_\p^\mu(m)$ is the $\bZ^2$-graded dg-algebra given by taking the quotient of $R_\p(m)$ by the two-sided ideal generated by the elements:
\[
\tikzdiagh{0}{
     	 \draw  (0,-.5) node[below] {\small $j$}  --(0,.5)node [midway,tikzdot]{} node[midway,xshift=-5ex,yshift=.75ex]{\small $\mu_j {-} \mu_{j+1}$}; 
      	\draw  (1,-.5) node[below] {\small $i_1$} -- (1,.5);
      	\node at(2,0) {\small $\dots$};
      	\draw  (3,-.5) node[below] {\small $i_{m-1}$} -- (3,.5);
 }
\]
for all $j \in I_\lv$, and grading
\begin{align*}
\deg \left(
\ 
\tikzdiag{
	\draw (0,0) node[below] {\small $i$}  ..controls (0,.5) and (1,.5) .. (1,1);
	\draw[myblue] (1,0) node[below] {\small $j$}  ..controls (1,.5) and (0,.5) .. (0,1);
}
\ 
\right)
&:= \begin{cases}
q^{-2}, &\text{if $i=j$}, \\
1, &\text{if $|i-j| > 1$}, \\
q, &\text{if $|i-j| = 1$},
\end{cases},
&
\deg \left(
\ 
\tikzdiag{
	\draw (0,0) node[below] {\small $i$}  -- (0,1) node [midway,tikzdot]{};
}
\ 
\right)
& := q^2,
\end{align*}
\begin{align*}
\deg\left(
\ 
\tikzdiag{
	\draw (0,0) node[below]{\small $j$} -- (0,1);
	\fdot{.5,.5};
}
\ 
\right)
&= q^{2(\mu_j - \mu_{j+1})} h,
\end{align*}
where $q^{k_1}\lambda^{k_2}h^{k_3}$ is our notation for degree $(k_1,k_2)$ for the $\bZ^2$-grading, and in degree $k_3$ for the homological grading (in particular, $1$ means it is in degree $0$ for all gradings). Note that  $\mu_j - \mu_{j+1}$ is in $\beta + \bZ$ for $j \in I_\np$ and so, a floating dot carries a non-trivial $\lambda$-degree. 
We denote $(R_\p^\mu(m), 0)$ the dg-algebra obtained by equipping $R_\p^\mu(m)$ with a trivial differential.
\end{defn}
Note that $R_\g^\mu(m)$ is the usual cyclotomic quotient of the KLR algebra, as in \cite{KL1}.

\subsection{Categorical $U_q(\g)$-action}

For $\nu = \sum_{i \in I_\g} \nu_i \cdot i$ with $\sum_i \nu_i = m$, we write $\Seq(\nu)$ for the set of sequences $\bi = i_1 i_2 \cdots i_m$ with $i_k \in I_\bg$ such that each $i \in I_\g$ appear exactly $\nu_i$ times in $\bi$. We write $\Seq(m)$ for the set of sequences $\bi = i_1 i_2 \cdots i_m$ with $i_k \in I_\bg$. 
For $\bi = i_1i_2\cdots i_m \in \Seq(m)$, we define the idempotent of $R_\p^\mu(m)$ given by
\[
1_\bi :=  \ 
\tikzdiagh[xscale=.75]{0}{
	\draw (0,0) node[below] {\plusspacing \small $i_1$} -- (0,1);
	\draw (1,0) node[below] {\plusspacing \small $i_2$} -- (1,1);
	\node at(2,.5) {$\dots$};
	\draw (3,0) node[below] {\plusspacing \small $i_{m}$} -- (3,1);
}
\]
We define $R_\p^\mu(\nu) := \bigoplus_{\bi,\bj \in \Seq(\nu)} 1_\bj R_\p^\mu(\nu) 1_\bi$, and $R_\p^\mu := \bigoplus_\nu R_\p^\mu(\nu)$. 

\smallskip

We consider categories $ (R_\p^\mu(m),0)\amod$ of $\bZ^2$-graded left dg-modules over $(R_\p^\mu(m),0)$. For such a (dg-)module $M$, we write $q^{k_1}\lambda^{k_2} M [k_3]$ for its grading shift up by $(k_1,k_2)$ in the $\bZ^2$ grading, and up by $k_3$ in the homological grading. Note that the grading shift in homological degree twists by a sign the action of $(R_\p^\mu(m),0)$, : $x \cdot m[1] := (-1)^{\deg_h(x) }(x \cdot m) [1]$.  

\smallskip

For each $i \in I_\g$, there is a non-unital map of dg-algebras $R_\p^\mu(\nu) \rightarrow R_\p^\mu(\nu+i)$ given by adding a vertical strand with label $i$ at the right:
\[
\tikzdiagh[xscale=.75]{0}{
	\draw (0,0) node[below] { \small $j_1$} -- (0,1);
	\draw (1,0) node[below] { \small $j_2$} -- (1,1);
	\draw (2,0) -- (2,1);
	\node at(2,-.25) {$\dots$};
	\draw (3,0) node[below] { \small $j_m$} -- (3,1);
	\filldraw [fill=white, draw=black] (-.25,.8) rectangle (3.25,0.2) node[midway] {$D$};
}
\ \mapsto \ 
\tikzdiagh[xscale=.75]{0}{
	\draw (0,0) node[below] { \small $j_1$} -- (0,1);
	\draw (1,0) node[below] { \small $j_2$} -- (1,1);
	\draw (2,0) -- (2,1);
	\node at(2,-.25) {$\dots$};
	\draw (3,0) node[below] { \small $j_m$} -- (3,1);
	\filldraw [fill=white, draw=black] (-.25,.8) rectangle (3.25,0.2) node[midway] {$D$};
	\draw (4,0) node[below] { \small $i$} -- (4,1);
}
\]
This gives rise to induction and restriction functors
\begin{align*}
\Ind_\nu^{\nu,i} : (R_\p^\mu(\nu),0)\amod \rightarrow (R_\p^\mu(\nu+i),0) \amod, \\
\Res_\nu^{\nu,i} : (R_\p^\mu(\nu+i),0) \amod \rightarrow  (R_\p^\mu(\nu),0) \amod, 
\end{align*}
which are adjoint. 
We put $\alpha_i(\nu) = 2\nu_i - \nu_{i-1} - \nu_{i+1}$. 
We define
\begin{align*}
\F_i &:= \bigoplus_\nu \Ind_\nu^{\nu,i},
&
\E_i &:=  \bigoplus_\nu q^{1+ \alpha_i(\nu) -(\mu_i - \mu_{i+1})} \Res_\nu^{\nu,i},
\end{align*}
where $q^\beta = \lambda$. 

\begin{prop}[{\cite[\S 5.4]{naissevaz3}}]
The endofunctors $\E_i : (R_\p^\mu,0) \rightarrow  (R_\p^\mu,0)$ and $\F_i : (R_\p^\mu,0) \rightarrow  (R_\p^\mu,0)$ are exact. 
\end{prop}

Let $\un_\nu$ be the identity functor on $(R_\p^\mu(\nu),0)\amod$. 
Let us also introduce the endofunctors
\[
\oplus_{[k]_q}(-) := \bigoplus_{\ell = 0}^{k-1} q^{1-k+2\ell} (-),
\]
and
\[
\oplus_{[\beta+k]_q}(-) := \bigoplus_{\ell \geq 0} q^{1+2\ell} \left( \lambda^{-1} q^{-k}(-) \oplus \lambda q^k (-)[1] \right).
\]

\begin{thm}[{\cite[Theorem 5.17 and Proposition 5.19]{naissevaz3}}]\label{thm:catActionVerma}
There is a natural short exact sequence
\begin{equation}\label{eq:vermaSES}
0 \rightarrow \F_i \E_i \un_\nu \rightarrow \E_i\F_i\un_\nu \rightarrow \oplus_{[(\mu_i-\mu_{i+1})-\alpha_i(\nu)]_q} \un_\nu \rightarrow 0,
\end{equation}
for all  $i \in I_\np$, 
and there are natural isomorphisms
\begin{equation}
\begin{split}
\E_i\F_i \un_\nu &\cong \F_i\E_i \un_\nu \oplus_{[(\mu_i-\mu_{i+1})-\alpha_i(\nu)]_q} \un_\nu,  \qquad \text{if $\mu_i-\mu_{i+1}-\alpha_i(\nu) \geq 0$}, \\
\F_i\E_i \un_\nu &\cong \E_i\F_i \un_\nu \oplus_{[\alpha_i(\nu)-(\mu_i-\mu_{i+1})]_q} \un_\nu,  \qquad \text{if $\mu_i-\mu_{i+1}-\alpha_i(\nu) \leq 0$}, 
\end{split}
\label{eq:irredDirectSum}
\end{equation}
for all $i \in I_\lv$. Furthermore, we have a natural isomorphism
\begin{equation}
\F_i\E_j \cong \E_j \F_i,
\label{eq:catFiEj}
\end{equation}
for all $i \neq j \in I_\g$. Finally, we have natural isomorphisms
\begin{align}
\begin{split}
\E_i^2\E_{i\pm1} \oplus \E_{i\pm1} \E_i^2 &\cong \oplus_{[2]_q} \E_i\E_{i\pm1}\E_i,
 \\
\F_i^2\F_{i\pm1} \oplus \F_{i\pm1} \F_i^2 &\cong \oplus_{[2]_q} \F_i\F_{i\pm1}\F_i, 
\end{split}
&&&
\begin{split}
\E_i\E_j &\cong \E_j \E_i,
\\
\F_i\F_j &\cong \F_j \F_i,
\end{split}
&\text{if $|i-j| > 1$}
\label{eq:catSerre}
\end{align}
for all $i,j \in I_\g$.
\end{thm}

Let us explain diagrammatically the maps involved in the short exact sequence \cref{eq:vermaSES}. For this,  we draw $R_\p^\mu(m)$ (viewed as $R_\p^\mu(m)$-$R_\p^\mu(m)$-bimodule) as a box labeled by $m$
\[
R_\p^\mu(m)
\ = \  
\tikzdiag[xscale=.75]{
	\draw (0,-.5) -- (0,.5);
	\draw (.5,-.5) -- (.5,.5);
	\draw (1.5,-.5) -- (1.5,.5);
	\draw (2,-.5) -- (2,.5);
	\node at(1,.4) {\small $\dots$};
	\node at(1,-.4) {\small $\dots$};
	\filldraw [fill=white, draw=black] (-.25,-.25) rectangle (2.25,0.25) node[midway] {\small $m$};
}
\]
and $\otimes_m := \otimes_{R_\p^\mu(m)}$ becomes stacking boxes on top of each other. We do the same for $R_\p^\mu(\nu)$. 
Moreover, we draw $\E_i R_\p^\mu(m+1)$ and $\F_iR_\p^\mu(m)$ respectively as
\begin{align*}
\tikzdiag[xscale=.75]{
	\draw (0,-.5) -- (0,.5);
	\draw (.5,-.5) -- (.5,.5);
	\draw (1.5,-.5) -- (1.5,.5);
	\draw (2,-.5) -- (2,.5);
	\draw (2.5, -.5) -- (2.5, .25) .. controls (2.5,.5) .. (2.75,.5) node[right]{\small $i$};
	\node at(1,.4) {\small $\dots$};
	\node at(1,-.4) {\small $\dots$};
	\filldraw [fill=white, draw=black] (-.25,-.25) rectangle (2.75,0.25) node[midway] {\small $m+1$};
} &&
\tikzdiag[xscale=.75,yscale=-1]{
	\draw (0,-.5) -- (0,.5);
	\draw (.5,-.5) -- (.5,.5);
	\draw (1.5,-.5) -- (1.5,.5);
	\draw (2,-.5) -- (2,.5);
	\draw (2.5, -.5) -- (2.5, .25) .. controls (2.5,.5) .. (2.75,.5) node[right]{\small $i$};
	\node at(1,.4) {\small $\dots$};
	\node at(1,-.4) {\small $\dots$};
	\filldraw [fill=white, draw=black] (-.25,-.25) rectangle (2.75,0.25) node[midway] {\small $m+1$};
}
\end{align*}
and so on. Also, the strands can be labeled by an element in $I_\g$, fixing an idempotent. 

\begin{rem}
In order to have a graded picture, we can say that:
\begin{align*}
\deg\left(\ 
\tikzdiag[xscale=.75]{
	\draw (0,0) node[below]{\small $j_1$} -- (0,.5);
	\draw (.5,0) node[below]{\small $j_2$} -- (.5,.5);
	\draw (1.5,0) -- (1.5,.5);
	\draw (2,0) node[below]{\small $j_m$} -- (2,.5);
	\draw (2.5, 0) node[below]{\small $i$} -- (2.5, .25) .. controls (2.5,.5) .. (2.75,.5);
	\node at(1,.25) {\small $\dots$};
}
\ \right) &=  
q^{1+\alpha_i(\bj)-(\mu_i-\mu_{i+1})},
&
\deg\left(\ 
\tikzdiag[xscale=.75,yscale=-1]{
	\draw (0,0) -- (0,.5);
	\draw (.5,0) -- (.5,.5);
	\draw (1.5,0) -- (1.5,.5);
	\draw (2,0) -- (2,.5);
	\draw (2.5, 0) -- (2.5, .25) .. controls (2.5,.5) .. (2.75,.5);
	\node at(1,.25) {\small $\dots$};
}
\ \right) &=  1,
\end{align*}
for all $\bj = j_1j_2\cdots j_m \in \Seq(m)$. 
\end{rem}

Then, the injection $\F_i\E_j \un_\nu \rightarrow \E_j\F_i\un_\nu$ in  \cref{eq:vermaSES}, as well as the similar maps in \cref{eq:irredDirectSum} and in \cref{eq:catFiEj}, are given by adding a crossing as follows:
\[
\tikzdiag[xscale=.75]{
	\draw (0,-1.25) -- (0,1.25);
	\draw (.5,-1.25) -- (.5,1.25);
	\draw (1.5,-1.25) -- (1.5,1.25);
	\draw (2,-1.25) -- (2,-.5) .. controls (2,-.25) .. (2.25,-.25) node[right]{\small $j$};
	\draw (2,1.25) -- (2,.5) .. controls (2,.25) .. (2.25,.25) node[right]{\small $i$};
	\node at(1,1.2) {\small $\dots$};
	\filldraw [fill=white, draw=black] (-.25,-1) rectangle (2.25,-.5) node[midway] {\small $m$};
	\node at(1,0) {\small $\dots$}; 
	\filldraw [fill=white, draw=black] (-.25,.5) rectangle (2.25,1) node[midway] {\small $m$};
	\node at(1,-1.2) {\small $\dots$};
}
\ \xmapsto{\ u_{ij} \ } \  
\tikzdiag[xscale=.75]{
	\draw (0,-1.25) -- (0,1.25);
	\draw (.5,-1.25) -- (.5,1.25);
	\draw (1.5,-1.25) -- (1.5,1.25);
	\draw (2,-1.25) -- (2,-.5) .. controls (2,0) and (2.5,0) .. (2.5,.5) -- (2.5,1.25)  node[above]{\small $j$};
	\draw (2,1.25) -- (2,.5) .. controls (2,0) and (2.5,0) .. (2.5,-.5) -- (2.5,-1.25)  node[below]{\small $i$};
	\node at(1,1.2) {\small $\dots$};
	\filldraw [fill=white, draw=black] (-.25,-1) rectangle (2.25,-.5) node[midway] {\small $m$};
	\node at(1,0) {\small $\dots$};
	\filldraw [fill=white, draw=black] (-.25,.5) rectangle (2.25,1) node[midway] {\small $m$};
	\node at(1,-1.2) {\small $\dots$};
}
\subset 
\tikzdiag[xscale=.65]{
	\draw (0,-1) -- (0,1);
	\draw (.5,-1) -- (.5,1);
	\draw (1.5,-1) -- (1.5,1);
	\draw (2,-1) -- (2,1);
	\draw (2.75,-1)  node [right] {\small $i$} .. controls (2.5,-1) ..
		(2.5,-.75) -- (2.5,.75)
		.. controls (2.5,1) .. (2.75,1) node [right] {\small $j$};
	\filldraw [fill=white, draw=black] (-.25,-.25) rectangle (2.75,.25) node[midway] {\small $m+1$};
	\node at(1,.65) {\small $\dots$};
	\node at(1,-.65) {\small $\dots$};
}
\]
Moreover, the map $\E_i\F_i \rightarrow \oplus_{[(\mu_i-\mu_{i+1})-\alpha_i(\nu)]_q} \un_\nu$ in \cref{eq:vermaSES} is given by projection on diagrams of the form:
\begin{equation}
\tikzdiag[xscale=.65]{
	\draw (0,-1) -- (0,1);
	\draw (.5,-1) -- (.5,1);
	\draw (1.5,-1) -- (1.5,1);
	\draw (2,-1) -- (2,1);
	\draw (2.75,-1)  node [right] {\small $i$} .. controls (2.5,-1) ..
		(2.5,-.75) -- (2.5,.75)
		.. controls (2.5,1) .. (2.75,1) node [right] {\small $i$};
	\filldraw [fill=white, draw=black] (-.25,-.25) rectangle (2.75,.25) node[midway] {\small $m+1$};
	\node at(1,.65) {\small $\dots$};
	\node at(1,-.65) {\small $\dots$};
}
\twoheadrightarrow
\bigoplus_{\ell \geq 0}
\tikzdiag[xscale=.65]{
	\draw (0,-1) -- (0,1);
	\draw (.5,-1) -- (.5,1);
	\draw (1.5,-1) -- (1.5,1);
	\draw (2,-1) -- (2,1);
	\draw (3,-1)  node [right] {\small $i$} .. controls (2.75,-1) ..
		(2.75,-.75) -- (2.75,.75) node[midway, tikzdot]{} node[midway,xshift=1.5ex, yshift=.5ex]{\small $\ell$}
		.. controls (2.75,1) .. (3,1) node [right] {\small $i$};
	\filldraw [fill=white, draw=black] (-.25,-.25) rectangle (2.25,.25) node[midway] {\small $m$};
	\node at(1,.65) {\small $\dots$};
	\node at(1,-.65) {\small $\dots$};
}
\oplus
\tikzdiag[xscale=.65]{
	\draw (0,-1) -- (0,-.75) .. controls (0,-.5) and (1,-.5) .. (1,-.25) .. controls (1,0) and (0,0) .. (0,.25)
		-- (0,1);
	\draw (.5,-1) -- (.5,-.75) .. controls (.5,-.5) and (1.5,-.5) .. (1.5,-.25) .. controls (1.5,0) and (.5,0) .. (.5,.25)
		-- (.5,1);
	\draw (1.5,-1) -- (1.5,-.75) .. controls (1.5,-.5) and (2.5,-.5) .. (2.5,-.25) .. controls (2.5,0) and (1.5,0) .. (1.5,.25)
		-- (1.5,1);
	\draw (2,-1) -- (2,-.75) .. controls (2,-.5) and (3,-.5) .. (3,-.25) .. controls (3,0) and (2,0) .. (2,.25)
		-- (2,1);
	\draw (3,-1)  node [right] {\small $i$} .. controls (2.75,-1) ..
		(2.75,-.75) .. controls (2.75,-.5) and (0,-.5) .. (0,-.25) node[pos=1,tikzdot]{} node[pos=1,xshift=-1.5ex]{\small $\ell$}
		 .. controls (0,0) and (2.75,0) .. (2.75,.25)
		-- (2.75,.75)
		.. controls (2.75,1) .. (3,1) node [right] {\small $i$};
 	\fdot{.6,-.25};
	\node at(1,.9) {\small $\dots$};
	\node at(1,-.9) {\small $\dots$};
	\node at(2,-.25) {\small $\dots$};
	\filldraw [fill=white, draw=black] (-.25,.25) rectangle (2.25,.75) node[midway] {\small $m$};
}
\label{eq:diagprojection}
\end{equation}

\begin{rem}
There exist functors $\E_i^{(a)}$ categorifying the action of the divided power $E_i^a/[a]_q!$, which are given exactly as in \cite[\S2.5]{KL1}. In particular, \cref{eq:catSerre} becomes
\begin{equation}
\begin{split}
\E_i^{(2)}\E_{i\pm1} \oplus \E_{i\pm1} \E_i^{(2)} &\cong \E_i\E_{i\pm1}\E_i, \\
\F_i^{(2)}\F_{i\pm1} \oplus \F_{i\pm1} \F_i^{(2)} &\cong \F_i\F_{i\pm1}\F_i,
\end{split}
\label{eq:catSerreDiv}
\end{equation}
for all $i \in I_\g$. 
\end{rem}

Since the results of \cref{thm:catActionVerma} need to take into consideration infinite direct sums, we need a refined version of Grothendieck group to decategorify $(R_\p^\mu,0)$. This can be done using the asymptotic Grothendieck groups $\bKO$, as introduced in \cite{asympK0}, and requiring $\Bbbk$ to be a field. 
Then, as explained in \cite[\S6]{naissevaz3}, one can take a certain subcategory $\cD^{lf}(R_\p^\mu,0)$ of the derived category of $(R_\p^\mu, 0)$, such that $\bKO(\cD^{lf}(R_\p^\mu,0)) \cong M^\p(\mu)$, as $U_q(\g)$-module with action of $E_i,F_i$ induced by $\E_i, \F_i$. 

\begin{rem}
One can define the functors $\E_i,\F_i$ using derived version of the induction and restriction functor instead. Conceptually, it would be more accurate. However, it requires a much more technically difficult framework to make sense of an exact triangle of functors (see \cite[\S 7]{naissevaz3}). 
\end{rem}

\subsection{Recovering cyclotomic KLR}\label{sec:cycloN}

For $\mu = \{\mu_i\}$, let $\overline \mu$ be given by specializing all $\beta$ to $N$ in $\mu_i$.  
Similarly, we can specialize the degree $\lambda$ in $R_\p^\mu(m)$ to $q^N$ for any $N \in \bZ$, giving a $\bZ$-graded dg-algebra $(R_\p^{\overline \mu}(m),0)$. Then, if $\overline \mu_j - \overline \mu_{j+1} \in \bN_0$ for all $j \in I_\np$, we can equip $R_\p^{\overline \mu}(m)$ with a non-trivial differential $d_N$ given by
\begin{align*}
d_N\left( \ 
\tikzdiag{
	\draw (0,0) node[below] {\small $i$}  ..controls (0,.5) and (1,.5) .. (1,1);
	\draw[myblue] (1,0) node[below] {\small $j$}  ..controls (1,.5) and (0,.5) .. (0,1);
}
\ \right)
=
d_N\left( \ 
\tikzdiag{
	\draw (0,0) node[below] {\small $i$}  -- (0,1) node [midway,tikzdot]{};
}
\ \right)
&= 0,
&
d_N\left( \ 
\tikzdiag{
	\draw (0,0) node[below]{\small $j$} -- (0,1);
	\fdot{.5,.5};
}
\ \right)
&= (-1)^{\overline \mu_j{-}\overline \mu_{j+1}}
\tikzdiag{
	\draw (0,0) node[below]{\small $j$} -- (0,1) node[midway,tikzdot]{} node[midway,xshift=5ex,yshift=.75ex]{\small $\overline \mu_j{-}\overline \mu_{j+1}$};
}
\end{align*} 
and extending using the graded Leibniz rule. 
It is not hard to check this is well defined.
 
\begin{thm}[{\cite[Theorem 5.20]{naissevaz3}}] \label{thm:formalRpN}
The $\bZ$-graded dg-algebra $(R_\p^{\overline\mu}(m), d_N)$ is formal with homology $H(R_\p^{\overline \mu}(m), d_N) \cong R_\g^{\overline \mu}(m)$. 
\end{thm}

Furthermore, by considering (derived) induction and restriction functors over $(R_p^{\overline \mu}(\nu),d_N)$, we obtain endofunctors $\E_i^N, \F_i^N$ on $\cD(R_\p^{\overline \mu}, d_N)$. 
The short exact sequence \cref{eq:vermaSES} becomes a short exact sequence of dg-bimodules with the infinite direct sum $(\oplus_{[(\mu_i-\mu_{i+1})-\alpha_i(\nu)_q]} \un_\nu, d_N)$ equipped with a non-trivial differential. This infinite direct sum is quasi-isomorphic to the finite direct sum $\oplus_{[(\overline \mu_i-\overline \mu_{i+1})-\alpha_i(\nu)]_q} (\un_\nu, 0)$. 
Also, the short exact sequence induces a long exact sequence in homology, which truncates and yields half the maps needed to construct the corresponding direct sum isomorphisms \cref{eq:irredDirectSum} for $R_\g^{\overline  \mu}$. See \cite[\S5.2]{naissevaz3} for more details.

%
%



\section{Link homology}\label{sec:cat}

We fix $\p$ and $\g$ as in \cref{ssec:linkpverma}, and highest weight $\beta = ((\beta)^n, (0)^n)$. 
We consider the $2$-category $\mathfrak M^\p( \beta)$ where objects are the categories $(R_\p^{\beta}, 0)\amod$ and hom spaces are categories of functors between them. 

\smallskip

Let $\dot{\mathcal{U}}(\mathfrak{sl}_n)$ denote  Khovanov--Lauda and Rouquier's 2-Kac--Moody algebra from~\cite{KL3,R1} 
(which are the same by~\cite{brundan}).  
The following result is immediate, thanks to \cref{thm:catActionVerma} and the fact that $\E_i$ and $\F_i$ are adjoint.
\begin{lem}\label{lem:KLR2KM-2M}
  There is a 2-action of $~\dot{\mathcal{U}}(\mathfrak{sl}_n\times\mathfrak{sl}_n)$
  on  $\tM^{\p}(\beta)$. 
\end{lem}

The lemma implies that, in particular, the categorified $q$-Schur algebra
$\mathcal{S}(0,n)$ from~\cite{MSV2} acts on  $\tM^{\p}(\beta)$.

\subsection{Braiding}\label{sec:braiding}

By a well-known construction due to Cautis~\cite{cautis-clasp},
we know how to associate  a 
chain complex in the 2-category $Kom(\dot{\mathcal{U}}(\sln))$
of complexes in the $\Hom$-categories of $\dot{\mathcal{U}}(\sln)$,
called a \emph{Rickard complex},  
which satisfies the braid relations up to homotopy.

\smallskip 

In our context, the Rickard complex is always truncated. For a positive crossing  between the $i$th and $(i+1)$th strands, it is given by the mapping cone
\begin{align*}
\sigma_i &\mapsto \cone\left( \F_i\E_i \un_\nu  \xrightarrow{\varepsilon_i} q^{-1} \un_\nu \right),
\intertext{where $\varepsilon_i$ is the counit of the adjunction $\un_{(1,1)}\F_i \dashv q^{-1}\E_i\un_{(1,1)}$. 
For a negative crossing, it is given by}
\sigma_i^{-1} &\mapsto \cone \left(q \un_\nu \xrightarrow{\eta_i} \E_i \F_i \un_\nu\right)[-1],
\end{align*}
with $\eta_i$ being the unit of the adjunction $\F_i\un_{(1,1)} \dashv q\un_{(1,1)}\E_i$. 

\begin{rem}
Note that for $i \neq 0$, $\E_i$ and $\F_i$ are biadjoint (up to degree shift). We also have $\F_i\E_i \un_{(1,1)} \cong \E_i\F_i \un_{(1,1)}$, and thus we can use the unit and counit for the other adjunction to build the mapping cone corresponding to the crossings.
\end{rem}

Diagrammatically, we can picture the maps $\varepsilon_i$ and $\eta_i$ as the following:
\begin{align*}
\tikzdiag[xscale=.75]{
	\draw (0,-1.25) -- (0,1.25);
	\draw (.5,-1.25) -- (.5,1.25);
	\draw (1.5,-1.25) -- (1.5,1.25);
	\draw (2,-1.25) -- (2,-.5) .. controls (2,-.25) .. (2.25,-.25) node[right]{\small $i$};
	\draw (2,1.25) -- (2,.5) .. controls (2,.25) .. (2.25,.25) node[right]{\small $i$};
	\node at(1,1.2) {\small $\dots$};
	\filldraw [fill=white, draw=black] (-.25,-1) rectangle (2.25,-.5) node[midway] {\small $\nu$};
	\node at(1,0) {\small $\dots$}; 
	\filldraw [fill=white, draw=black] (-.25,.5) rectangle (2.25,1) node[midway] {\small $\nu$};
	\node at(1,-1.2) {\small $\dots$};
}
\ &\xmapsto{\ \varepsilon_i \ } \  
\tikzdiag[xscale=.75]{
	\draw (0,-1.25) -- (0,1.25);
	\draw (.5,-1.25) -- (.5,1.25);
	\draw (1.5,-1.25) -- (1.5,1.25);
	\draw (2,-1.25) -- (2,-.5) .. controls (2,-.25) .. (2.25,-.25) .. controls (2.5,-.25) .. (2.5,0)
	.. controls (2.5,.25) .. (2.25,.25)  .. controls (2,.25) ..   (2,.5)--(2,1.25);
	\node at(1,1.2) {\small $\dots$};
	\filldraw [fill=white, draw=black] (-.25,-1) rectangle (2.25,-.5) node[midway] {\small $\nu$};
	\node at(1,0) {\small $\dots$};
	\filldraw [fill=white, draw=black] (-.25,.5) rectangle (2.25,1) node[midway] {\small $\nu$};
	\node at(1,-1.2) {\small $\dots$};
}
\ \subset \  
\tikzdiag[xscale=.65]{
	\draw (0,-1) -- (0,1);
	\draw (.5,-1) -- (.5,1);
	\draw (1.5,-1) -- (1.5,1);
	\draw (2,-1) -- (2,1);
	\filldraw [fill=white, draw=black] (-.25,-.25) rectangle (2.25,.25) node[midway] {\small $\nu$};
	\node at(1,.65) {\small $\dots$};
	\node at(1,-.65) {\small $\dots$};
}
\intertext{and}
\tikzdiag[xscale=.75]{
	\draw (0,-.5) -- (0,.5);
	\draw (.5,-.5) -- (.5,.5);
	\draw (1.5,-.5) -- (1.5,.5);
	\draw (2,-.5) -- (2,.5);
	\node at(1,.4) {\small $\dots$};
	\node at(1,-.4) {\small $\dots$};
	\filldraw [fill=white, draw=black] (-.25,-.25) rectangle (2.25,0.25) node[midway] {\small $\nu$};
}
\ &\xmapsto{ \ \eta_i \ } \ 
\tikzdiag[xscale=.75]{
	\draw (0,-.5) -- (0,.5);
	\draw (.5,-.5) -- (.5,.5);
	\draw (1.5,-.5) -- (1.5,.5);
	\draw (2,-.5) -- (2,.5);
	\draw (2.75,-.5) node[right]{\small $i$} .. controls (2.5,-.5) .. (2.5, -.25) -- (2.5, .25) .. controls (2.5,.5) .. (2.75,.5) node[right]{\small $i$};
	\node at(1,-.4) {\small $\dots$};
	\node at(1,.4) {\small $\dots$};
	\filldraw [fill=white, draw=black] (-.25,-.25) rectangle (2.25,0.25) node[midway] {\small $\nu$};
}
 \ \subset  \ 
\tikzdiag[xscale=.75]{
	\draw (0,-.5) -- (0,.5);
	\draw (.5,-.5) -- (.5,.5);
	\draw (1.5,-.5) -- (1.5,.5);
	\draw (2,-.5) -- (2,.5);
	\draw (2.75,-.5) node[right]{\small $i$} .. controls (2.5,-.5) .. (2.5, -.25) -- (2.5, .25) .. controls (2.5,.5) .. (2.75,.5) node[right]{\small $i$};
	\node at(1,-.4) {\small $\dots$};
	\node at(1,.4) {\small $\dots$};
	\filldraw [fill=white, draw=black] (-.25,-.25) rectangle (2.75,0.25) node[midway] {\small $\nu+i$};
}
\end{align*}

\smallskip

Following Cautis's construction in~\cite{cautis-clasp}, we associate a Rickard complex 
$C'(b)$ in the 2-category $Kom(\tM^{\p}(\beta))$
of complexes in the  
$\Hom$-categories of $\tM^{\p}(\beta)$,
to each braid diagram $b$ on $n$ strands. This gives a braiding on
the homotopy category 
of $Kom(\tM^{\p}(\beta))$.

\subsection{Invariance under the Markov moves}
 
Closing the diagram for $b$ consists of 
precomposing $C'(b)$ with the appropriate word on functors $\F_{-n+1},\dotsc , \F_{n-1}$,
and composing it with the appropriate word from $\E_{-n+1},\dotsc , \E_{n-1}$,
following the patterns in~\cref{eq:closingbraid}. 
This results in a chain complex $C'(\cl(b))$ in $Kom(\tM^{\p}(\beta))$,
which is a complex of endofunctors of the block 
corresponding to the highest weight in $M^{\p}( (\beta)^n,(0)^n )$, that is,
a complex of $\bZ^2$-graded $\Bbbk$-vector spaces.

\begin{lem}\label{lem:ladderisotopy}
  The homotopy type of the chain complex $C'(\cl(b))$ is invariant under isotopy of ladder diagrams:
\[
\tikz[very thick, yscale=.75, baseline={([yshift=.8ex]current bounding box.center)}]{
  \draw [dotted] (1,-.95)  -- (1,.1);
  \draw  (1,-1.5) node[below]{\small $1$} -- (1,-1);
\draw [->] (1,0)  -- (1,.5);
\draw [dotted] (1.75,-1.5)  node[below]{\small $0$}   -- (1.75,-1);
\draw [dotted] (1.75,0)   -- (1.75,.5);
\draw  (1.75,-1)   -- (1.75,0);
\draw [directed=.6] (1,-1) -- (1.75,-1); 
\draw [directed=.6] (1.75,0) -- (1,0); 
\node at (0.5,-0.5) {$\cdots$};
\node at (2.5,-0.5) {$\cdots$};
}
\ \cong \ 
\tikz[very thick, yscale=.75, baseline={([yshift=.8ex]current bounding box.center)}]{
\draw [->] (1,-1.5) node[below]{\small $1$}   -- (1,.5);
\draw [dotted] (1.75,-1.5)  node[below]{\small $0$}   -- (1.75,.5);
\node at (0.5,-0.5) {$\cdots$};
\node at (2.5,-0.5) {$\cdots$};
}, 
\qquad
\tikz[very thick,xscale=-1, yscale=.75, baseline={([yshift=.8ex]current bounding box.center)}]{
  \draw [dotted] (1,-.95)  -- (1,.1);
  \draw  (1,-1.5) node[below]{\small $1$} -- (1,-1);
\draw [->] (1,0)  -- (1,.5);
\draw [dotted] (1.75,-1.5)  node[below]{\small $0$}   -- (1.75,-1);
\draw [dotted] (1.75,0)   -- (1.75,.5);
\draw  (1.75,-1)   -- (1.75,0);
\draw [directed=.6] (1,-1) -- (1.75,-1); 
\draw [directed=.6] (1.75,0) -- (1,0); 
\node at (0.5,-0.5) {$\cdots$};
\node at (2.5,-0.5) {$\cdots$};
}
\ \cong \ 
\tikz[very thick,xscale=-1, yscale=.75, baseline={([yshift=.8ex]current bounding box.center)}]{
\draw [->] (1,-1.5) node[below]{\small $1$}   -- (1,.5);
\draw [dotted] (1.75,-1.5)  node[below]{\small $0$}   -- (1.75,.5);
\node at (0.5,-0.5) {$\cdots$};
\node at (2.5,-0.5) {$\cdots$};
}
\]
\end{lem}
\begin{proof}
These are straightforward consequences of~\cref{eq:irredDirectSum}, since they give $\E_i\F_i\un_{(1,0)} \cong \un_{(1,0)}$ and $\F_i\E_i\un_{(0,1)} \cong \un_{(0,1)}$. 
\end{proof}

\begin{prop}\label{prop:CmarkovI}
  The homotopy type of the chain complex $C'(\cl(b))$ is invariant under the Markov of type I.
\end{prop}

\begin{proof}
We want to show that
\begin{equation}\label{eq:Bi-stepleft}
  \raisebox{-38pt}{
    \reflectbox{
\begin{tikzpicture}
  \draw [very thick] (0.25,-1.7) node[below]{\reflectbox{\small $1$}} -- (0.25,-1.4);
  \draw [very thick] (0.25,-.9) -- (0.25,0);
  \draw [very thick,dotted] (0.25,.05)  -- (0.25,.5);
  \draw [very thick,dotted] (1,-.45)  -- (1,.1);
  \draw [very thick] (1,-.9)  -- (1,-.5);
  \draw [very thick] (1,-1.7) node[below]{\reflectbox{\small $1$}} -- (1,-1.4);
  \draw [very thick,directed=.6] (0.25,0) -- (1,0); 
  \draw [very thick] (0.15,-.9) -- (1.1,-.9); 
  \draw [very thick] (1.1,-1.4) -- (0.15,-1.4); 
  \draw [very thick] (1.1,-1.4)  -- (1.1,-.9);
  \draw [very thick] (0.15,-1.4)  -- (0.15,-.9);
  \node at (.61,-1.16) {\reflectbox{\small $\sigma_i^{\pm 1}$}};
\draw [very thick,->] (1,0)  -- (1,.5);
\draw [very thick,dotted] (1.75,-1.7)   -- (1.75,-.5);
\draw [very thick,->] (1.75,-.5)   -- (1.75,.5);
\draw [very thick,directed=.6] (1,-.5) -- (1.75,-.5); 
\node at (-.5,-0.8) {$\cdots$};
\node at (2.5,-0.8) {$\cdots$};
\end{tikzpicture} }}
\cong
\raisebox{-38pt}{
\reflectbox{
  \begin{tikzpicture}
  \draw [very thick] (0.25,-1.7) node[below]{\reflectbox{\small $1$}} -- (0.25,-.7);
  \draw [very thick,dotted] (0.25,.5)  -- (0.25,-.65);
  \draw [very thick] (1,-.7)  -- (1,-.3);
  \draw [very thick] (1,-1.7) node[below]{\reflectbox{\small $1$}} -- (1,-1.2);
  \draw [very thick,directed=.6] (0.25,-.7) -- (1,-.7); 
  \draw [very thick] (.9,-.3) -- (1.85,-.3); 
  \draw [very thick] (1.85,.2) -- (.9,.2); 
  \draw [very thick] (1.85,-.3)  -- (1.85,.2);
  \draw [very thick] (.9,-.3)  -- (.9,.2);
  \node at (1.4,-.06) {\reflectbox{\small $\sigma_{i-1}^{\pm 1}$}};
\draw [very thick,->] (1,.2)  -- (1,.5);
\draw [very thick,dotted] (1,-1.15)   -- (1,-.75);
\draw [very thick,dotted] (1.75,-1.7)   -- (1.75,-1.23);
\draw [very thick] (1.75,-1.2)   -- (1.75,-.3);
\draw [very thick,->] (1.75,.2)   -- (1.75,.5);
\draw [very thick,directed=.6] (1,-1.2) -- (1.75,-1.2); 
\node at (-.5,-0.8) {$\cdots$};
\node at (2.5,-0.8) {$\cdots$};
\end{tikzpicture} }} ,
\end{equation}
where $\sigma_i^{\pm 1}$ denote the mapping cones defined in \cref{sec:braiding}, and similarly for downward oriented strands, and for the bottom part of the braid closure. 
Then,
the first Markov move  can be decomposed in a sequence of moves of the following form:
\begin{equation}\label{eq:Bitraceup}
\raisebox{-38pt}{
\reflectbox{
  \begin{tikzpicture}
  \draw [very thick] (0.25,-.9)  -- (0.25,0);
  \draw [very thick,dotted] (0.25,.05)  -- (0.25,.5);
  \draw [very thick,dotted] (1,-.45)  -- (1,-.05);
  \draw [very thick] (1,-.9)  -- (1,-.5);
  \draw [very thick,directed=.6] (0.25,0) -- (1,0); 
  \draw [very thick] (1,0)  -- (1,.5);
\draw [very thick,directed=.6] (1,.5) -- (1.75,.5); 
\draw [very thick] (1.75,0)  -- (1.75,.5);
\draw [very thick] (1.75,-1.7)node[below]{\reflectbox{\small $-1$}}  -- (1.75,-.5);
\draw [very thick,dotted] (1.75,-.45)  -- (1.75,-.05);
\draw [very thick] (2.5,-1.7) node[below]{\reflectbox{\small $-1$}}  -- (2.5,0);
\draw [very thick,dotted] (2.5,.05)  -- (2.5,.5);
\draw [very thick,directed=.6] (1.75,0) -- (2.5,0); 
\draw [very thick,directed=.6] (1,-.5) -- (1.75,-.5); 
  \draw [very thick] (0.15,-.9) -- (1.1,-.9); 
  \draw [very thick] (1.1,-1.4) -- (0.15,-1.4); 
  \draw [very thick] (1.1,-1.4)  -- (1.1,-.9);
  \draw [very thick] (0.15,-1.4)  -- (0.15,-.9);
  \node at (.61,-1.16) {\reflectbox{\small $\sigma_1^{\pm 1}$}};
 \draw [very thick] (0.25,-1.7) node[below]{\reflectbox{\small $1$}}  -- (0.25,-1.4);
 \draw [very thick] (1,-1.7) node[below]{\reflectbox{\small $1$}}  -- (1,-1.4);  
\node at (-.5,-0.8) {$\cdots$};
\node at (3.25,-0.8) {$\cdots$};
\end{tikzpicture} }}
\,\cong\  
\raisebox{-38pt}{
  \reflectbox{
\begin{tikzpicture}
  \draw [very thick] (0.25,-1.7) node[below]{\reflectbox{\small $1$}}  -- (0.25,0);
  \draw [very thick,dotted] (0.25,.05)  -- (0.25,.5);
  \draw [very thick,dotted] (1,-.45)  -- (1,-.05);
  \draw [very thick] (1,-1.7) node[below]{\reflectbox{\small $1$}} -- (1,-.5);
  \draw [very thick,directed=.6] (0.25,0) -- (1,0); 
  \draw [very thick] (1,0)  -- (1,.5);
\draw [very thick,directed=.6] (1,.5) -- (1.75,.5); 
\draw [very thick] (1.75,0)  -- (1.75,.5);
\draw [very thick] (1.75,-1.7) node[below]{\reflectbox{\small $-1$}}  -- (1.75,-1.4);
\draw [very thick,dotted] (1.75,-.45)  -- (1.75,-.05);
\draw [very thick] (2.5,-1.7) node[below]{\reflectbox{\small $-1$}}  -- (2.5,-1.4);
\draw [very thick,dotted] (2.5,.05)  -- (2.5,.5);
\draw [very thick,directed=.6] (1.75,0) -- (2.5,0); 
\draw [very thick,directed=.6] (1,-.5) -- (1.75,-.5); 
  \draw [very thick] (1.65,-.9) -- (2.6,-.9); 
  \draw [very thick] (2.6,-1.4) -- (1.65,-1.4); 
  \draw [very thick] (2.6,-1.4)  -- (2.6,-.9);
  \draw [very thick] (1.65,-1.4)  -- (1.65,-.9);
  \node at (2.12,-1.16) {\reflectbox{\small $\sigma_{-1}^{\pm 1}$}};
 \draw [very thick] (1.75,-.9) -- (1.75,-.5);
 \draw [very thick] (2.5,-.9)  -- (2.5,0);  
\node at (-.5,-0.8) {$\cdots$};
\node at (3.25,-0.8) {$\cdots$};
\end{tikzpicture} }} ,
\end{equation}
(to avoid cluttering we have dropped the $\beta$'s from the pictures, 
since it is clear where to place them), 
and similar for the bottom part of the closure. 

Relation~\cref{eq:Bi-stepleft} requires an isomorphism of 1-morphisms in $\tM^{\p}(\beta)$
\begin{equation}\label{eq:EiEimEiFiiso}
\E_i\E_{i-1}\E_i\F_i\un_{(\dotsc ,0,1,1, \dotsc )} \cong \E_{i-1}\F_{i-1}\E_i\E_{i-1}\un_{(\dotsc ,0,1 ,1,\dotsc)}
\end{equation}
which is proved in~\cite[Lemma 3.19]{QR1}, after applying~\cref{lem:ladderisotopy}, and using \cref{lem:KLR2KM-2M}. 
Moreover, the computations in~\cite[Lemma 3.19]{QR1} also implies that the diagrams 
\[
\begin{tikzcd}[column sep=-4ex]
\E_i\E_{i-1}\E_i\F_i\un_{(\dotsc ,0,1,1, \dotsc )}
\ar{rr}{\eqref{eq:EiEimEiFiiso}}
& &
 \E_{i-1}\F_{i-1}\E_i\E_{i-1}\un_{(\dotsc ,0,1 ,1,\dotsc)}
 \\
 & q\E_i\E_{i-1}\un_{(\dotsc ,0,1,1, \dotsc )} \ar[leftarrow]{ul}{\varepsilon_i} \ar[swap,leftarrow]{ur}{\varepsilon_{i-1}}  &
\end{tikzcd}
\]
commute, and thus we obtain~\cref{eq:Bi-stepleft} for $\sigma_i$. The case for $\sigma_i^{-1}$ is similar.

We write $\E_\pm$ instead of $\E_{\pm 1}$, and the same for $\F_\pm$. 
To prove relation~\cref{eq:Bitraceup} we write 
\begin{align*}
\Le\un_{(\beta-1,\beta-1,1,1)} &:= \E_0\E_+\E_-\E_0\F_+\E_+\un_{(\beta-1,\beta-1,1,1)}
\\
\R\un_{(\beta-1,\beta-1,1,1)} &:= \E_0\E_+\E_-\E_0\F_-\E_-\un_{(\beta-1,\beta-1,1,1)} 
\end{align*}
By \cref{eq:catFiEj} we have the following isomorphism
\begin{align*}\allowdisplaybreaks
\Le\un_{(\beta-1,\beta-1,1,1)} &= \E_0\E_+\E_-\E_0\F_+\E_+\un_{(\beta-1,\beta-1,1,1)}
\\ &
\cong \E_0\E_+\F_+  \E_-\E_0\E_+\un_{(\beta-1,\beta-1,1,1)} ,
\end{align*}
and by \cref{eq:irredDirectSum} we have
\begin{align*}
\E_+\F_+ \un_{(\beta,\beta-1,1,0)} &\cong  \un_{(\beta,\beta-1,1,0)}  \oplus \F_+\E_+ \un_{(\beta,\beta-1,1,0)} \\
&\cong \un_{(\beta,\beta-1,1,0)} 
\end{align*}
since $\E_+ \un_{(\beta,\beta-1,1,0)} \cong 0$ by weight reasons. 
Therefore, we obtain
\begin{align*}
\Le\un_{(\beta-1,\beta-1,1,1)} &
\cong \E_0\E_-\E_0\E_+\un_{(\beta-1,\beta-1,1,1)} 
\\ &
\overset{\eqref{eq:catSerreDiv}}{\cong} \E_0^{(2)}\E_-\E_+\un_{(\beta-1,\beta-1,1,1)},
\end{align*}
since again $\E_0^{(2)}\E_+\un_{(\beta-1,\beta-1,1,1)} \cong 0$ by weight reasons.
Similarly, we obtain
\begin{align*}
\R\un_{(\beta-1,\beta-1,1,1)} &= \E_0\E_+\E_-\E_0\F_-\E_-\un_{(\beta-1,\beta-1,1,1)}
\\ &
\cong \E_0\E_+\E_0\E_-\un_{(\beta-1,\beta-1,1,1)}
\\ &
\cong \E_0^{(2)}\E_+\E_-\un_{(\beta-1,\beta-1,1,1)} .
\end{align*}
Thus, $\Le\un_{(\beta-1,\beta-1,1,1)}  \cong \R\un_{(\beta-1,\beta-1,1,1)}$. 
Moreover, since the isomorphisms $\Le\un_{(\beta-1,\beta-1,1,1)} \cong  \E_0^{(2)}\E_-\E_+\un_{(\beta-1,\beta-1,1,1)}$ and $\R\un_{(\beta-1,\beta-1,1,1)} \cong \E_0^{(2)}\E_+\E_-\un_{(\beta-1,\beta-1,1,1)} $ are obtained from similar operations on diagrams (exchanging the role of the colors $1$ and $-1$), it means we obtain a commutative diagram:
\[
\begin{tikzcd}[column sep=-4ex]
\Le\un_{(\beta-1,\beta-1,1,1)} 
\ar{rr}{\simeq}
&&
\R\un_{(\beta-1,\beta-1,1,1)}
\\
&
\E_0\E_+\E_-\E_0\un_{(\beta-1,\beta-1,1,1)} 
\ar[leftarrow]{ul}{\varepsilon_+}
\ar[leftarrow, swap]{ur}{\varepsilon_-}
&
\end{tikzcd}
\]
Thus, we obtain the wanted isomorphism in \cref{eq:Bitraceup}.
The proofs for the bottom part and for $\sigma_{\pm 1}^{-1}$ are similar. 
\end{proof}

\begin{prop}\label{prop:CmarkovII}
  The homotopy type of the chain complex $C'(\cl(b))$ is invariant under the Markov of type II, up to a global $\lambda$-degree shift. 
\end{prop}

\begin{proof}
Consider diagrams $D_0$ and $D_1^+$ that differ as below:
\begin{align*}
  D_1^+ &= 
  \raisebox{-20pt}{
    \reflectbox{
\begin{tikzpicture}[scale=.3]
        \draw [very thick] (1.5,-1.5)-- (-1.5,1.5);
	\draw [very thick] (0.5,0.5)  -- (1.5,1.5);
	\draw [very thick] (-1.5,-1.5)  -- (-0.5,-0.5);
        \draw [very thick,directed=.6] (4.5,1.5 ) -- (4.5,-1.5);
	\draw [very thick] (4.5,-1.5)  -- (1.5,-1.5);
        \draw [very thick] (4.5, 1.5)  -- (1.5, 1.5);
 \draw [very thick, dotted] (1.5,1.8) -- (1.5,2.4);
 \draw [very thick, dotted] (4.5,1.8) -- (4.5,2.4);
 \draw [very thick,->] (-1.5,1.5) -- (-1.5,2.4);
 \draw [very thick, dotted] (1.5,-2.4)  -- (1.5,-1.8);
 \draw [very thick, dotted] (4.5,-2.4)  -- (4.5,-1.8);
 \draw [very thick] (-1.5,-2.4) -- (-1.5,-1.5);
\end{tikzpicture}}}
&
D_0 &=
\raisebox{-20pt}{
  \reflectbox{
\begin{tikzpicture}[scale=.3]
	\draw [very thick,->] (-1.5,-2.4) -- (-1.5,2.4);
	\draw [very thick,dotted] (1.5,-2.4) -- (1.5,2.4);
       	\draw [very thick,dotted] (4.5,-2.4) -- (4.5,2.4);
\end{tikzpicture}}}
\end{align*}
The complex for $D_1^+$ is
 \[
 C'(D_1^+) \cong \cone\left( \E_0 \F_{1}\E_{1} \F_0\un_{(\beta-0,0,1)}
 \xrightarrow{\ \ \varepsilon_1 \ \ }  q^{-1}\E_0\F_0\un_{(\beta-0,0,1)} \right),
\]
where we need to think of $(\beta-0,0,1)$ as living inside a bigger $\nu$ depending on the global diagram.

We obtain an isomorphism 
(note that in this case $\F_0\E_0\un_{(\beta-0,1,0)}$ is zero)  
\begin{equation}
\begin{split}
  \E_0 \F_{1}\E_{1} \F_0\un_{(\beta-0,0,1)}
  &\overset{\eqref{eq:catFiEj}}{\cong}  \F_1\E_0\F_0\E_1 \un_{(\beta-0,0,1)}
  \\
  &\overset{\eqref{eq:vermaSES}}{\cong} \oplus_{[\beta-1]_q}  \F_1\E_1 \un_{(\beta-0,0,1)}
  \\
  &\overset{\eqref{eq:irredDirectSum}}{\cong}   q^{2}\lambda^{-1} \bigl( \Bbbk[\xi]  \otimes \un_{(\beta-0,0,1)}\bigr) \oplus  \lambda\bigl(\Bbbk[\xi] \otimes \un_{(\beta-0,0,1)}\bigr)[1],  
  \end{split}
  \label{eq:isoE0F1E1F0}
\end{equation}
with $\deg(\xi) = q^2$. 

This means we can think of the diagrams in $\E_0 \F_{1}\E_{1} \F_0\un_{(\beta-0,0,1)}$ as being all of the form:
\begin{align}
&\tikzdiag[xscale=.65]{
	\draw (0,-1) -- (0,1);
	\draw (.5,-1) -- (.5,1);
	\draw (1.5,-1) -- (1.5,1);
	\draw (2,-1) -- (2,1);
	\draw (3.25,-1)  node [right] {\small $i$} .. controls (3,-1) ..
		(3,-.75) .. controls (3,-.5) and (2.5,-.5) .. (2.5,-.25) -- (2.5,.25)node[midway, tikzdot]{} node[midway,xshift=1.5ex, yshift=.5ex]{\small $\ell$}
		 .. controls (2.5,.5) and (3,.5) .. 
		(3,.75) .. controls (3,1) .. (3.25,1) node [right] {\small $i$};
	\filldraw [fill=white, draw=black] (-.25,-.25) rectangle (2.25,.25) node[midway] {\small $\nu$};
	\node at(1,.65) {\small $\dots$};
	\node at(1,-.65) {\small $\dots$};
	\draw[myblue] (2.5,-1) node[below]{\small $1$} -- (2.5,-.75) .. controls (2.5,-.5) .. (3.25,-.25);
	\draw[myblue] (2.5,1) node[above]{\small $1$} -- (2.5,.75) .. controls (2.5,.5) .. (3.25,.25);
}
&
&\tikzdiag[xscale=.65]{
	\draw (0,-1) -- (0,-.75) .. controls (0,-.5) and (1,-.5) .. (1,-.25) .. controls (1,0) and (0,0) .. (0,.25)
		-- (0,1);
	\draw (.5,-1) -- (.5,-.75) .. controls (.5,-.5) and (1.5,-.5) .. (1.5,-.25) .. controls (1.5,0) and (.5,0) .. (.5,.25)
		-- (.5,1);
	\draw (1.5,-1) -- (1.5,-.75) .. controls (1.5,-.5) and (2.5,-.5) .. (2.5,-.25) .. controls (2.5,0) and (1.5,0) .. (1.5,.25)
		-- (1.5,1);
	\draw (2,-1) -- (2,-.75) .. controls (2,-.5) and (3,-.5) .. (3,-.25) .. controls (3,0) and (2,0) .. (2,.25)
		-- (2,1);
	\draw (3,-1) node[below]{\small $i$} .. controls (2.75,-.5) and (0,-.5) .. (0,-.25) node[pos=1,tikzdot]{} node[pos=1,xshift=-1.5ex]{\small $\ell$}
		 .. controls (0,0) and (3,0) .. (3,.25)
		--  (3,.75) 
		.. controls (3,1) .. (3.25,1) node [right] {\small $i$};
	\draw[myblue] (2.5,-1) node[below]{\small $1$} .. controls (2.5,-.75) .. (3.25,-.5);
	\draw[myblue] (3.25,0) .. controls (2.5,.25) .. (2.5,.5) -- (2.5,1) node[above]{\small $1$};
 	\fdot{.6,-.25};
	\node at(1,.9) {\small $\dots$};
	\node at(1,-.9) {\small $\dots$};
	\node at(2,-.25) {\small $\dots$};
	\filldraw [fill=white, draw=black] (-.25,.25) rectangle (2.25,.75) node[midway] {\small $\nu$};
}
\label{eq:basisE0F1E1F0}
\end{align}
for $\ell \geq 0$, and where $\un_{(\beta-0,0,1)} \cong R_\p^\beta(\nu)$. Moreover, we identify $\xi^\ell$ with a dot with label $\ell$. 

Applying the short exact sequence~\cref{eq:vermaSES}  to the term on the right gives 
\[
q^{-1}\E_0\F_0\un_{(\beta-0,0,1)} \cong \lambda^{-1} \bigl(\Bbbk[\xi] \otimes \un_{(\beta-0,0,1)}\bigr) \oplus  \lambda  \bigl(\Bbbk[\xi] \otimes \un_{(\beta-0,0,1)}\bigr)[1].
\]
In terms of diagrams, we get from \cref{eq:diagprojection} that $\E_0\F_0\un_{(\beta-0,0,1)}$ has a basis given by:
\begin{align}
&\tikzdiag[xscale=.65]{
	\draw (0,-1) -- (0,1);
	\draw (.5,-1) -- (.5,1);
	\draw (1.5,-1) -- (1.5,1);
	\draw (2,-1) -- (2,1);
	\draw (3.25,-1)  node [right] {\small $i$} .. controls (3,-1) ..
		(3,-.75) -- (3,.75)node[midway, tikzdot]{} node[midway,xshift=1.5ex, yshift=.5ex]{\small $\ell$}
		..controls (3,1) .. (3.25,1) node [right] {\small $i$};
	\filldraw [fill=white, draw=black] (-.25,-.25) rectangle (2.25,.25) node[midway] {\small $\nu$};
	\node at(1,.65) {\small $\dots$};
	\node at(1,-.65) {\small $\dots$};
	\draw[myblue] (2.5,-1) node[below]{\small $1$} -- (2.5,1);
}
&
&\tikzdiag[xscale=.65]{
	\draw (0,-1) -- (0,-.75) .. controls (0,-.5) and (1,-.5) .. (1,-.25) .. controls (1,0) and (0,0) .. (0,.25)
		-- (0,1);
	\draw (.5,-1) -- (.5,-.75) .. controls (.5,-.5) and (1.5,-.5) .. (1.5,-.25) .. controls (1.5,0) and (.5,0) .. (.5,.25)
		-- (.5,1);
	\draw (1.5,-1) -- (1.5,-.75) .. controls (1.5,-.5) and (2.5,-.5) .. (2.5,-.25) .. controls (2.5,0) and (1.5,0) .. (1.5,.25)
		-- (1.5,1);
	\draw (2,-1) -- (2,-.75) .. controls (2,-.5) and (3,-.5) .. (3,-.25) .. controls (3,0) and (2,0) .. (2,.25)
		-- (2,1);
	\draw (3,-1) node[below]{\small $i$} .. controls (2.75,-.5) and (0,-.5) .. (0,-.25) node[pos=1,tikzdot]{} node[pos=1,xshift=-1.5ex]{\small $\ell$}
		 .. controls (0,0) and (3,0) .. (3,.25)
		--  (3,.75) 
		.. controls (3,1) .. (3.25,1) node [right] {\small $i$};
	\draw[myblue] (2.5,-1) node[below]{\small $1$} -- (2.5,-.75) .. controls (2.5,-.5) and (3.5,-.5) .. (3.5,-.25)
		.. controls (3.5,0) and (2.5,0) .. (2.5,.25) -- (2.5,1) node[above]{\small $1$};
 	\fdot{.6,-.25};
	\node at(1,.9) {\small $\dots$};
	\node at(1,-.9) {\small $\dots$};
	\node at(2,-.25) {\small $\dots$};
	\filldraw [fill=white, draw=black] (-.25,.25) rectangle (2.25,.75) node[midway] {\small $\nu$};
}
\label{eq:basisE0F0}
\end{align} 
for all $\ell \geq 0$.

Because of \cref{eq:KLRR2}, applying $\varepsilon_1$ on \cref{eq:basisE0F1E1F0} gives that 
 $C'(D_1^+)$ is isomorphic to the complex 
\[
C'(D_1^+) \cong
\begin{pmatrix}
 q^2 \lambda^{-1} (\Bbbk[\xi] \otimes \un_{(\beta-0,0,1)}) \\[1ex] \lambda (\Bbbk[\xi] \otimes \un_{(\beta-0,0,1)})[1] 
\end{pmatrix} 
\xrightarrow{\ \ \psi\ \ }
\begin{pmatrix}
 \lambda^{-1}( \Bbbk[\xi] \otimes \un_{(\beta-0,0,1)}) \\[1ex]\lambda (\Bbbk[\xi] \otimes \un_{(\beta-0,0,1)})[1] 
\end{pmatrix} ,
\]
with differential 
\[
\psi :=\begin{pmatrix}
   \xi \otimes 1 + 1 \otimes f & 0
  \\
  0 & 1,
\end{pmatrix}.
\]
for some map $f : q^2 \un_{(\beta-0,0,1)} \rightarrow \un_{(\beta-0,0,1)}$. 
After the removal of all acyclic subcomplexes, we get that $C'(D_1^+)$ is homotopy equivalent to the complex 
\[
C'(D_1^+) \cong 0\to \lambda^{-1}\Bbbk\un_{(\beta-0,0,1)}.
\]
Therefore, we have that the complexes $C'(D_1^+)$ and
$\lambda^{-1} C'(D_0)$
are homotopy equivalent.

\smallskip

Similarly for  a diagram $D_1^-$ containing a negative crossing, we first show that
\begin{align*}
 \E_0 \E_1 \F_1 \F_0 \un_{(\beta-0,0,1)}[-1] &\overset{\eqref{eq:irredDirectSum}}{\cong}  \E_0 \F_1 \E_1 \F_0 \un_{(\beta-0,0,1)}[-1] 
 \\
 &\overset{\eqref{eq:isoE0F1E1F0}}{\cong} q^2 \lambda^{-1} (\Bbbk[\xi] \otimes \un_{(\beta-0,0,1)}) [-1] \oplus \lambda (\Bbbk[\xi] \otimes \un_{(\beta-0,0,1)}.
\end{align*}
This means that $ \E_0 \E_1 \F_1 \F_0 \un_{(\beta-0,0,1)}[-1]$ is given by diagrams of the form:

\begin{align}
&
\tikzdiag[xscale=.65]{
	\draw (0,-1) -- (0,1);
	\draw (.5,-1) -- (.5,1);
	\draw (1.5,-1) -- (1.5,1);
	\draw (2,-1) -- (2,1);
	\draw (3.25,-1)  node [right] {\small $i$} .. controls (3,-1) ..
		(3,-.75) -- (3,-.5)
		.. controls (3,-.25) and (2.5,-.25) .. (2.5,0) node[pos=1, tikzdot]{} node[pos=1,xshift=1.25ex, yshift=0ex]{\small $\ell$}
		.. controls (2.5,.25) and (3,.25) .. 
		 (3,.5) -- (3,.75)
		..controls (3,1) .. (3.25,1) node [right] {\small $i$};
	\draw[myblue] (3.75,-.75)  node [right] {\small $1$} .. controls (3.5,-.75) ..(3.5,-.5) 
		.. controls (3.5,.25) and (2.5,.25) .. 
		(2.5,.5)  -- (2.5,1);
	\draw[myblue] (2.5,-1) node[below]{\small $1$}  -- (2.5,-.5)
		.. controls (2.5,-.25) and (3.5,-.25) ..
		(3.5,.5) ..controls (3.5,.75) .. (3.75,.75) node [right] {\small $1$};
	\filldraw [fill=white, draw=black] (-.25,-.25) rectangle (2.25,.25) node[midway] {\small $\nu$};
	\node at(1,.65) {\small $\dots$};
	\node at(1,-.65) {\small $\dots$};
}
&
&
\tikzdiag[xscale=.65]{
	\draw (0,-1) -- (0,-.75) .. controls (0,-.5) and (1,-.5) .. (1,-.25) .. controls (1,0) and (0,0) .. (0,.25)
		-- (0,1);
	\draw (.5,-1) -- (.5,-.75) .. controls (.5,-.5) and (1.5,-.5) .. (1.5,-.25) .. controls (1.5,0) and (.5,0) .. (.5,.25)
		-- (.5,1);
	\draw (1.5,-1) -- (1.5,-.75) .. controls (1.5,-.5) and (2.5,-.5) .. (2.5,-.25) .. controls (2.5,0) and (1.5,0) .. (1.5,.25)
		-- (1.5,1);
	\draw (2,-1) -- (2,-.75) .. controls (2,-.5) and (3,-.5) .. (3,-.25) .. controls (3,0) and (2,0) .. (2,.25)
		-- (2,1);
	\draw (3,-1) node[below]{\small $i$} .. controls (2.75,-.5) and (0,-.5) .. (0,-.25) node[pos=1,tikzdot]{} node[pos=1,xshift=-1.5ex]{\small $\ell$}
		 .. controls (0,0) and (3,0) .. (3,.25)
		--  (3,.75) 
		.. controls (3,1) .. (3.25,1) node [right] {\small $i$};
	\draw[myblue] (2.5,-1) node[below]{\small $1$} -- (2.5,-.75)
		.. controls (2.5,-.5) and (4,-.5) .. 
		 (4,.5) ..controls (4,.75) .. (4.25,.75) node [right] {\small $1$};
	\draw[myblue] (4.25,-.75)  node [right] {\small $1$} .. controls (4,-.75) ..(4,-.5)
		.. controls (4,0) and (2.5,0) ..
		 (2.5,.25) -- (2.5,1) node[above]{\small $1$};
 	\fdot{.6,-.25};
	\node at(1,.9) {\small $\dots$};
	\node at(1,-.9) {\small $\dots$};
	\node at(2,-.25) {\small $\dots$};
	\filldraw [fill=white, draw=black] (-.25,.25) rectangle (2.25,.75) node[midway] {\small $\nu$};
}
\label{eq:basisE0E1F1F0}
\end{align} 
for all $\ell \geq 0$. 
Moreover, we observe that in $ \E_0 \E_1 \F_1 \F_0 \un_{(\beta-0,0,1)}[-1]$ we have
\begin{align}
\tikzdiag[xscale=.65]{
	\draw (0,-1) -- (0,1);
	\draw (.5,-1) -- (.5,1);
	\draw (1.5,-1) -- (1.5,1);
	\draw (2,-1) -- (2,1);
	\draw (3.25,-1)  node [right] {\small $i$} .. controls (3,-1) ..
		(3,-.75) -- (3,.75)node[midway, tikzdot]{} node[midway,xshift=1ex, yshift=.75ex]{\small $\ell$}
		..controls (3,1) .. (3.25,1) node [right] {\small $i$};
	\draw[myblue] (3.75,-.75)  node [right] {\small $1$} .. controls (3.5,-.75) ..
		(3.5,-.5) -- (3.5,.5)
		..controls (3.5,.75) .. (3.75,.75) node [right] {\small $1$};
	\filldraw [fill=white, draw=black] (-.25,-.25) rectangle (2.25,.25) node[midway] {\small $\nu$};
	\node at(1,.65) {\small $\dots$};
	\node at(1,-.65) {\small $\dots$};
	\draw[myblue] (2.5,-1) node[below]{\small $1$} -- (2.5,1);
}
&=
\tikzdiag[xscale=.65]{
	\draw (0,-1) -- (0,1);
	\draw (.5,-1) -- (.5,1);
	\draw (1.5,-1) -- (1.5,1);
	\draw (2,-1) -- (2,1);
	\draw (3.25,-1)  node [right] {\small $i$} .. controls (3,-1) ..
		(3,-.75) -- (3,-.5)
		.. controls (3,-.25) and (2.5,-.25) .. (2.5,0) node[pos=1, tikzdot]{} node[pos=1,xshift=1.25ex, yshift=0ex]{\small $\ell$}
		.. controls (2.5,.25) and (3,.25) .. 
		 (3,.5) -- (3,.75)
		..controls (3,1) .. (3.25,1) node [right] {\small $i$};
	\draw[myblue] (3.75,-.75)  node [right] {\small $1$} .. controls (3.5,-.75) ..(3.5,-.5) 
		.. controls (3.5,.25) and (2.5,.25) .. 
		(2.5,.5)  -- (2.5,1);
	\draw[myblue] (2.5,-1) node[below]{\small $1$}  -- (2.5,-.5)
		.. controls (2.5,-.25) and (3.5,-.25) ..
		(3.5,.5) ..controls (3.5,.75) .. (3.75,.75) node [right] {\small $1$};
	\filldraw [fill=white, draw=black] (-.25,-.25) rectangle (2.25,.25) node[midway] {\small $\nu$};
	\node at(1,.65) {\small $\dots$};
	\node at(1,-.65) {\small $\dots$};
}
\label{eq:1b1tob11}
\end{align}
because of \cref{eq:KLRR3}, and the fact that two consecutive strands labeled by $1$ must be zero at this position (this is a consequence of the fact that $\F_1\F_1 \un_{(\beta-0,1,0)} \cong 0$ by weight reasons).
We also have
\begin{align*}
\tikzdiag[xscale=.65]{
	\draw (0,-1) -- (0,-.75) .. controls (0,-.5) and (1,-.5) .. (1,-.25) .. controls (1,0) and (0,0) .. (0,.25)
		-- (0,1);
	\draw (.5,-1) -- (.5,-.75) .. controls (.5,-.5) and (1.5,-.5) .. (1.5,-.25) .. controls (1.5,0) and (.5,0) .. (.5,.25)
		-- (.5,1);
	\draw (1.5,-1) -- (1.5,-.75) .. controls (1.5,-.5) and (2.5,-.5) .. (2.5,-.25) .. controls (2.5,0) and (1.5,0) .. (1.5,.25)
		-- (1.5,1);
	\draw (2,-1) -- (2,-.75) .. controls (2,-.5) and (3,-.5) .. (3,-.25) .. controls (3,0) and (2,0) .. (2,.25)
		-- (2,1);
	\draw (3,-1) node[below]{\small $i$} .. controls (2.75,-.5) and (0,-.5) .. (0,-.25) node[pos=1,tikzdot]{} node[pos=1,xshift=-1.5ex]{\small $\ell$}
		 .. controls (0,0) and (3,0) .. (3,.25)
		--  (3,.75) 
		.. controls (3,1) .. (3.25,1) node [right] {\small $i$};
	\draw[myblue] (2.5,-1) node[below]{\small $1$} -- (2.5,-.75) .. controls (2.5,-.5) and (3.5,-.5) .. (3.5,-.25)
		.. controls (3.5,0) and (2.5,0) .. (2.5,.25) -- (2.5,1) node[above]{\small $1$};
	\draw[myblue] (4.25,-.75)  node [right] {\small $1$} .. controls (4,-.75) ..
		(4,-.5) -- (4,.5)
		..controls (4,.75) .. (4.25,.75) node [right] {\small $1$};
 	\fdot{.6,-.25};
	\node at(1,.9) {\small $\dots$};
	\node at(1,-.9) {\small $\dots$};
	\node at(2,-.25) {\small $\dots$};
	\filldraw [fill=white, draw=black] (-.25,.25) rectangle (2.25,.75) node[midway] {\small $\nu$};
}
&\overset{\eqref{eq:1b1tob11}}{=}
\tikzdiag[xscale=.65]{
	\draw (0,-1) -- (0,-.75) .. controls (0,-.5) and (1,-.5) .. (1,-.25) .. controls (1,0) and (0,0) .. (0,.25)
		-- (0,1);
	\draw (.5,-1) -- (.5,-.75) .. controls (.5,-.5) and (1.5,-.5) .. (1.5,-.25) .. controls (1.5,0) and (.5,0) .. (.5,.25)
		-- (.5,1);
	\draw (1.5,-1) -- (1.5,-.75) .. controls (1.5,-.5) and (2.5,-.5) .. (2.5,-.25) .. controls (2.5,0) and (1.5,0) .. (1.5,.25)
		-- (1.5,1);
	\draw (2,-1) -- (2,-.75) .. controls (2,-.5) and (3,-.5) .. (3,-.25) .. controls (3,0) and (2,0) .. (2,.25)
		-- (2,1);
	\draw (3,-1) node[below]{\small $i$} .. controls (2.75,-.5) and (0,-.5) .. (0,-.25) node[pos=1,tikzdot]{} node[pos=1,xshift=-1.5ex]{\small $\ell$}
		 .. controls (0,0) and (3,0) .. 
		 (3,.25)
		  -- 
		 (3,.75) 
		.. controls (3,1) .. (3.25,1) node [right] {\small $i$};
	\draw[myblue] (2.5,-1) node[below]{\small $1$} -- (2.5,-.75) .. controls (2.5,-.5) and (3.5,-.5) .. (3.5,-.25)
		.. controls (3.5,0) and (2.5,0) .. (2.5,.25)
		.. controls (2.5,.375) .. (4.25,.75) node [right] {\small $1$};
%
	\draw[myblue] (4.25,-.75)  node [right] {\small $1$} .. controls (4,-.75) ..
		(4,-.5) -- (4,.25)
		.. controls (4,.625) and (2.5,.625) .. (2.5,1) node[above]{\small $1$};
%
 	\fdot{.6,-.25};
	\node at(1,.9) {\small $\dots$};
	\node at(1,-.9) {\small $\dots$};
	\node at(2,-.25) {\small $\dots$};
	\filldraw [fill=white, draw=black] (-.25,.25) rectangle (2.25,.75) node[midway] {\small $\nu$};
}
\equiv
\tikzdiag[xscale=.65]{
	\draw (0,-1) -- (0,-.75) .. controls (0,-.5) and (1,-.5) .. (1,-.25) .. controls (1,0) and (0,0) .. (0,.25)
		-- (0,1);
	\draw (.5,-1) -- (.5,-.75) .. controls (.5,-.5) and (1.5,-.5) .. (1.5,-.25) .. controls (1.5,0) and (.5,0) .. (.5,.25)
		-- (.5,1);
	\draw (1.5,-1) -- (1.5,-.75) .. controls (1.5,-.5) and (2.5,-.5) .. (2.5,-.25) .. controls (2.5,0) and (1.5,0) .. (1.5,.25)
		-- (1.5,1);
	\draw (2,-1) -- (2,-.75) .. controls (2,-.5) and (3,-.5) .. (3,-.25) .. controls (3,0) and (2,0) .. (2,.25)
		-- (2,1);
	\draw (3,-1) node[below]{\small $i$} .. controls (2.75,-.5) and (0,-.5) .. (0,-.25) node[pos=1,tikzdot]{} node[pos=1,xshift=-3ex]{\small $\ell{+}1$}
		 .. controls (0,0) and (3,0) .. (3,.25)
		--  (3,.75) 
		.. controls (3,1) .. (3.25,1) node [right] {\small $i$};
	\draw[myblue] (2.5,-1) node[below]{\small $1$} -- (2.5,-.75)
		.. controls (2.5,-.5) and (4,-.5) .. 
		 (4,.5) ..controls (4,.75) .. (4.25,.75) node [right] {\small $1$};
	\draw[myblue] (4.25,-.75)  node [right] {\small $1$} .. controls (4,-.75) ..(4,-.5)
		.. controls (4,0) and (2.5,0) ..
		 (2.5,.25) -- (2.5,1) node[above]{\small $1$};
 	\fdot{.6,-.25};
	\node at(1,.9) {\small $\dots$};
	\node at(1,-.9) {\small $\dots$};
	\node at(2,-.25) {\small $\dots$};
	\filldraw [fill=white, draw=black] (-.25,.25) rectangle (2.25,.75) node[midway] {\small $\nu$};
}
\end{align*} 
where the symbol $\equiv$ means equality up to adding terms with less than $\ell+1$ dots next to the floating dot, or of the form in~\cref{eq:basisE0E1F1F0} left. 
This equality follows from applying~\cref{eq:KLRR2}, and then a sequence of \cref{eq:KLRdotslide} and \cref{eq:KLRnh} to slide down the newly spawned dot on the $i$-strand. 
Furthermore,  as before, $q \E_0 \F_0 \un_{(\beta-0,0,1)}[-1]$ is given by the same diagrams as in \cref{eq:basisE0F0}. Thus, applying $\eta_1$ on them gives an isomorphism of complexes
\begin{align*}
D_1^- &=
q \E_0 \F_0 \un_{(\beta-0,0,1)}[-1] \xrightarrow{\eta_1} \E_0 \E_1 \F_1 \F_0 \un_{(\beta-0,0,1)}[-1] 
 \\ 
&\cong \begin{pmatrix}
  q^2 \lambda^{-1} (\Bbbk[\xi]  \otimes \un_{(\beta-0,0,1)}) [-1] \\[1ex]  q^2 \lambda ( \Bbbk[\xi] \otimes \un_{(\beta-0,0,1)})  
\end{pmatrix} 
\xrightarrow{\ \ \bar \psi  \ \ }
\begin{pmatrix}
 q^2 \lambda^{-1} (\Bbbk[\xi] \otimes \un_{(\beta-0,0,1)}) [-1] \\[1ex]\lambda (\Bbbk[\xi] \otimes \un_{(\beta-0,0,1)} )
\end{pmatrix},
\end{align*}
where
\[
\bar \psi := \begin{pmatrix}1 & g' \\ 0 & \xi \otimes 1 + 1 \otimes g \end{pmatrix},
\]
for some $g :  q^2 \un_{(\beta-0,0,1)} \rightarrow \un_{(\beta-0,0,1)}$ and $g' : q^2 \lambda ( \Bbbk[\xi] \otimes \un_{(\beta-0,0,1)})   \rightarrow  q^2 \lambda^{-1} (\Bbbk[\xi] \otimes \un_{(\beta-0,0,1)}) [-1]$. 
The last complex is homotopy equivalent to the complex
\begin{align*}
0 \to   \lambda \un_{(\beta-0,0,1)}, 
\end{align*}
so that
$C'(D_1^-)$ and $\lambda C'(D_0)$ are homotopy equivalent.
\end{proof}

Define the \emph{normalized} chain complex 
$C(\cl(b)):= \lambda^{n_+ - n_-} C'(\cl(b))$
where as usual, $n_\pm$ is the number of positive/negative crossings in $\cl(b)$.

\begin{cor}\label{cor:eulerchar}
  The homology groups $H(b)$ of $C( \cl(b) )$ are triply-graded link invariants, 
  and their bigraded Euler characteristic is the HOMFLY-PT polynomial of the closure of $b$. 
\end{cor}

\subsection{$H(b)$ is isomorphic to the Khovanov-Rozansky HOMFLY-PT link homology}

We now show that our link homology is equivalent to the
HOMFLY-PT link homology by Khovanov and Rozansky~\cite{KR2,Kh}, 
by proving that $H(b)$ is isomorphic to Rasmussen's version of
HOMFLY-PT homology in~\cite{rasmussen}.  

\begin{thm}\label{thm:HisoKR}
  For every braid $b$ the homology, 
  $H(b)$ is isomorphic to Khovanov-Rozansky HOMFLY-PT link homology $\HKR(b)$, after regrading. 
\end{thm}

The theorem above gives us an equivalence in a weak sense.
We conjecture the equivalence is in fact stronger, in the following sense. 
The Soergel category $\mathcal{SC}_n$ from~\cite{EK} acts on $\tM^{\p}(\beta)$, in particular, on its
$( (\beta-1)^n, (1)^n )$-weight space 
(this action goes through the categorified $q$-Schur algebra
$\mathcal{S}(n,n)$, see~\cite[\S6]{MSV2}, which also acts on $\tM^{\p}(\beta)$, as 
explained at the start of~\cref{sec:cat}.
Composing with the operation of closing the braid on the top with the correct sequence of $\E_i$'s,
following the pattern in~\cref{eq:closingbraid}, gives a functor $\Phi$
from $\mathcal{SC}_n$ to the category $Ab_{h,q,\lambda}$ of triply-graded abelian groups. 
Note that the same pattern, but with $\F_i$'s,
has to be used on the bottom of the diagram to create the weight on which $\mathcal{SC}_n$ acts.

\begin{conj}
  The functor $\Phi$ is isomorphic to the Hochschild homology functor. 
\end{conj}

\medskip

We now prove~\cref{thm:HisoKR}.
We assume the reader is familiar with~\cite{rasmussen}. 
Let $L$ be a link presented as the closure of a braid $b$ in $n$ strands.
Recall that the process of closing $b$ amounts to composing 
a word in $\E_i$'s with the Rickard complex for $b$ 
(after adding $n$ parallel strands at its right) and with a word in $\F_i$'s. 
Of course, the closure procedure extends canonically to webs.  
Let $S$ be the (polynomial) ring in the dots on the $\F_0$'s used to form the closure of
a web $\Gamma$.

\begin{lem}\label{lem:Hgfree}
For every web $\Gamma$, $H(\Gamma)$ is a free module over $S$. 
\end{lem}

\begin{proof}
  The proof follows the same reasoning as the proof of Rasmussen of an analogous result using
  matrix factorizations~\cite[Proposition 4.8]{rasmussen} which is based on an induction scheme
  introduced by Wu~\cite[\S3]{wu}.
  The only thing we need to check are the MOY relations $\mathrm{0}$ to $\mathrm{III}$
  from \cite[\S4.2]{rasmussen}.  
  MOY relations $\mathrm{II}$ and $\mathrm{III}$ are already satisfied in
 $\dot{\mathcal{U}}(\sln)$, and
  MOY relations $\mathrm{0}$ and $\mathrm{I}$ are a direct consequence of the
  short exact sequence~\cref{eq:vermaSES}, when applied to the weights
  $(\dotsc, \beta-0, 0,\dotsc)$ and $(\dotsc, \beta-0, 1, \dotsc)$, 
  since one of the terms in the exact sequence always act as the zero functor on these weights.  
\end{proof}

\begin{proof}[Proof of~\cref{thm:HisoKR}]
  Since both our construction and the one in~\cite[Proposition 4.8]{rasmussen} satisfy the MOY relations, the underlying spaces of the complexes $C'(cl(b))$ and of $\HKR(b)$ are isomorphic by \cref{lem:Hgfree}. Moreover, the braiding in both constructions is the Rickard complex 
, and thus the complexes are equivalent. 
The regrading is given by identifying the $(i,j,k)$-grading of~\cite{rasmussen} as $q = i, \lambda = j$ and $h = (j-k)/2$. 
\end{proof}

\subsection{Khovanov-Rozansky's $\glN$-link homologies} \label{ssec:KRGLN}

Using the 2-representation of $\glnn$, con\-struc\-ted from the cyclotomic KLR algebra 
$R_\g^{( (N)^{n}, (0)^{n} )}$ as input instead of a 2-category similar to $\tM^{\p}(\beta)$, results in
Khovanov and Rozansky's $\glN$-link homology $H_N(L)$ from~\cite{KR1}. 
This follows at once from the work of Mackaay and Yo\-ne\-zawa~\cite{MY}.

\subsection{Reduced homologies}\label{sec:reducedhomology}

A modification of our construction could be used to give a construction of the reduced version of KR HOMFLY-PT link homology. 
Using the parabolic subalgebra $\rp\subseteq U_q(\mathfrak{gl}_{2n-1})$ and the highest weight 
$\bar \beta = ((\beta)^{n-1},1,(0)^{n-1})$ gives a 2-Verma module $\mathfrak{M}^{\rp}(\bar \beta)$. Constructing Rickard complexes with it should result in reduced versions $\overline{H}$ of $H$.
All the results in the preceding subsections have analogues for the case of reduced homology,
and should be proven essentially in the same way as above. 
However, there is one subtlety to take into account when claiming the equivalence with reduced
Khovanov--Rozansky homology.  
Recall that in the case of the reduced versions 
from~\cite{KR1,KR2,rasmussen} 
the abelian groups $\overline{H}(L,i)$ and $\overline{H}_N(L,i)$ are invariants of the link $L$  
together with a marked component $i$.
With our choice of cutting out and open a diagram of a braid closure in~\cref{ssec:normalized},
the outermost strand in our version (the one that is cut) corresponds to the preferred component $i$ of $\cl(b)$
(as described in~\cite{rasmussen})
under the isomorphism between our reduced homologies and Khovanov--Rozansky's.
Using the cyclotomic  KLR algebra $R^{( (N)^{n-1},1,(0)^{n-1} )}$ for $\mathfrak{gl}_{2n-1}$ should result in
a reduced version $\overline{H}_N$ of Khovanov--Rozansky's $\glN$-link homology $H_N$.

\subsection{HOMFLY-PT to $\glN$ spectral sequence}\label{sec:spectralSequence}

We explain how to construct a spectral sequence from HOMFLY-PT homology to $\gl_N$-homology in our context, for $N > 0$, akin to Rasmussen's spectral sequence in~\cite{rasmussen}.  

Recall from \cref{sec:cycloN} that the functors $\E_i^N$ and $\F_i^N$ are given by derived induction and restriction. Thus, they are adjoint and give rise to Rickard complexes of dg-bimodules (in other words, the maps $\eta_i$ and $\varepsilon_i$ from \cref{sec:braiding} are maps of dg-bimodules). Thus, we obtain a bicomplex $(C(L),d_r,d_N)$, where $d_r$ is the Rickard differential. 

To any bicomplex $(M,d',d'')$ we can associate two spectral sequences 
$\{E_\ell^I,d_\ell^I\}$ and $\{E_\ell^{II},d_\ell^{II}\}$,
which are induced by the two canonical filtrations.
Moreover, we have that $E_2^I = H( H(M,d''), d')$ and $E_2^{II} = H( H(M,d'), d'')$, and 
if the double complex is bounded, then both spectral sequences converge to the total homology
$H(M,d'+d'')$. 
We will also use the fact that if $H(M,d'')$ (resp. $H(M,d')$) is concentrated in a single $d''$-degree (resp. $d'$-degree), then $E^I$ (resp. $E^{II}$) converges at the second page, meaning that $E_2^I \cong H(M,d''+d')$ (resp. $E_2^{II} \cong H(M,d'+d'')$).

For a link $L$ presented in the form of a closure of a braid $b$, we form the bounded double complex
$( C(L),d_r, d_N )$. 
Let $\{E_\ell^I,d_\ell^I\}$ and $\{E_\ell^{II},d_\ell^{II}\}$ be respectively the spectral sequences induced by the $N$-filtration
and $r$-filtration.

Recall that a strongly projective  (see \cite{sixdgmodels} or \cite{naissevaz3} for a precise definition) left $(A,d_A)$-dg-module $(P,d_P)$ is such that for any right $(A,d_A)$-dg-module $(M,d_M)$ we have
\[
H\bigl((M,d_M) \otimes_{(A,d_A)} (P,d_P) \bigr) \cong H(M,d_M) \otimes_{H(A,d_A)} H(P,d_P).
\]
By \cite[Proposition 5.15]{naissevaz3}, we know that  $(R_\p^\beta(\nu+i),d_N)$ is strongly projective as $(R_\p^\beta(\nu),d_N)$-module. Thus, \cref{thm:formalRpN} tells us that $H(C(L),d_N)$ is  concentrated in a single $d_N$-degree. 
As a consequence, $E^{I}$ converges at the second page. Thus, we know that 
\[
H(C(L),d_N+d_r) \cong H(H(C(L),d_N), d_r) \cong \HKR_N(L),
\] 
thanks to~\cref{ssec:KRGLN}.
Thus, $\{E_r^I,d_r^I\}$ is a spectral sequence whose $E_1$-page is $H(C(L), d_r) \cong \HKR(L)$, 
which converges to $\HKR_N(L)$. 

Note that the spectral sequence in~\cite{rasmussen} is constructed for the reduced case, and that we can also introduce a $d_N$ on the reduced homology in~\cref{sec:reducedhomology} to fall in the same case. Then both spectral sequences share similar properties: they start from the same underlying spaces (up to isomorphism), with a $E_1$-page being (reduced) $\HKR(L)$ in one direction and converging at the second page in the other direction to (reduced) $\HKR_N(L)$.  

\subsection{Colored link homology}

A version of $\tM^{\p}(\beta)$ for divided powers of the $\F_i$'s and $\E_i$'s 
could be used to construct a version of HOMFLY-PT homology for links colored
by minuscule representations of $\mathfrak{gl}_N$, as the one constructed by Mackaay-Sto\v{s}i\'c-Vaz 
in~\cite{MSV} and Webster-Williamson~\cite{WW}. 
Moreover, the differential $d_N$ would give rise to a spectral sequence to colored 
$\mathfrak{gl}_N$-Khovanov--Rozansky link homology, as the one constructed by Wedrich in~\cite{W}.  
However, proving a version of the first exact sequence from~\cref{thm:catActionVerma} for divided powers 
might be a nontrivial problem.

%
%



\bibliographystyle{bibliography/habbrv}


\end{document}